\documentclass[a4paper,twoside]{article}
\usepackage{a4}
\usepackage{amssymb}
\usepackage{amsmath}
\usepackage{upref}
\usepackage{url}
\usepackage{xypic}
\usepackage{mathtools}
\usepackage[active]{srcltx}
\usepackage[dvipsnames]{color}
\allowdisplaybreaks[2] 
\newcount\minutes \newcount\hours
\hours=\time
\divide\hours 60
\minutes=\hours
\multiply\minutes -60
\advance\minutes \time
\newcommand{\klockan}{\the\hours:{\ifnum\minutes<10 0\fi}\the\minutes}
\newcommand{\tid}{\today\ \klockan}
\newcommand{\prtid}{\smash{\raise 10mm \hbox{\LaTeX ed \tid}}}
\renewcommand{\prtid}{}

\def\vint{\mathop{\mathchoice%
          {\setbox0\hbox{$\displaystyle\intop$}\kern 0.22\wd0%
           \vcenter{\hrule width 0.6\wd0}\kern -0.82\wd0}%
          {\setbox0\hbox{$\textstyle\intop$}\kern 0.2\wd0%
           \vcenter{\hrule width 0.6\wd0}\kern -0.8\wd0}%
          {\setbox0\hbox{$\scriptstyle\intop$}\kern 0.2\wd0%
           \vcenter{\hrule width 0.6\wd0}\kern -0.8\wd0}%
          {\setbox0\hbox{$\scriptscriptstyle\intop$}\kern 0.2\wd0%
           \vcenter{\hrule width 0.6\wd0}\kern -0.8\wd0}}%
          \mathopen{}\int}

\makeatletter
\def\sectionmark#1{} 
\def\subsectionmark#1{}
\newcommand{\sectnr}{\ifnum \c@secnumdepth >\z@
                 \thesection.\hskip 1em\relax \fi}
\def\@evenhead{\footnotesize\rm\thepage\hfil\leftmark\hfil\llap{\prtid}}
\def\@oddhead{\footnotesize\rm\rlap{\prtid}\hfil\rightmark\hfil\thepage}
\def\tableofcontents{\section*{Contents} 
 \@starttoc{toc}}
\makeatother
\makeatletter
\def\@biblabel#1{#1.}
\makeatother
\makeatletter
\let\Thebibliography=\thebibliography
\renewcommand{\thebibliography}[1]{\def\@mkboth##1##2{}\Thebibliography{#1}
\addcontentsline{toc}{section}{References}
\frenchspacing 
\setlength{\@topsep}{0pt}
\setlength{\itemsep}{0pt}%
\setlength{\parskip}{0pt plus 2pt}%
}
\makeatother
\makeatletter
\def\mdots@{\mathinner.\nonscript\!.%
 \ifx\next,.\else\ifx\next;.\else\ifx\next..\else
 \nonscript\!\mathinner.\fi\fi\fi}
\let\ldots\mdots@
\let\cdots\mdots@
\let\dotso\mdots@
\let\dotsb\mdots@
\let\dotsm\mdots@
\let\dotsc\mdots@
\def\vdots{\vbox{\baselineskip2.8\p@ \lineskiplimit\z@
    \kern6\p@\hbox{.}\hbox{.}\hbox{.}\kern3\p@}}
\def\ddots{\mathinner{\mkern1mu\raise8.6\p@\vbox{\kern7\p@\hbox{.}}%
    \raise5.8\p@\hbox{.}\raise3\p@\hbox{.}\mkern1mu}}
\makeatother
\makeatletter
\def\@seccntformat#1{\csname the#1\endcsname.\quad}
\makeatother
\makeatletter
\long\def\@makecaption#1#2{%
  \vskip\abovecaptionskip
  \sbox\@tempboxa{ #1. #2}%
  \ifdim \wd\@tempboxa >\hsize
    #1. #2\par
  \else
    \global \@minipagefalse
    \hb@xt@\hsize{\hfil\box\@tempboxa\hfil}%
  \fi
  \vskip\belowcaptionskip}
\makeatother
\makeatletter
\renewcommand\section{\@startsection {section}{1}{\z@}%
                                   {-3.5ex \@plus -1ex \@minus -.2ex}%
                                   {2.3ex \@plus.2ex}%
                                   {\normalfont\Large\bfseries\boldmath}}
\renewcommand\subsection{\@startsection{subsection}{2}{\z@}%
                                     {-3.25ex\@plus -1ex \@minus -.2ex}%
                                     {1.5ex \@plus .2ex}%
                                     {\normalfont\large\bfseries\boldmath}}
\renewcommand\subsubsection{\@startsection{subsubsection}{3}{\z@}%
                                     {-3.25ex\@plus -1ex \@minus -.2ex}%
                                     {1.5ex \@plus .2ex}%
                                     {\normalfont\normalsize\bfseries\boldmath}}
\renewcommand\paragraph{\@startsection{paragraph}{4}{\z@}%
                                    {3.25ex \@plus1ex \@minus.2ex}%
                                    {-1em}%
                                    {\normalfont\normalsize\bfseries\boldmath}}
\renewcommand\subparagraph{\@startsection{subparagraph}{5}{\parindent}%
                                       {3.25ex \@plus1ex \@minus .2ex}%
                                       {-1em}%
                                      {\normalfont\normalsize\bfseries\boldmath}}
\makeatother
\newcommand{\authortitle}[2]{\author{#1}\title{#2}\markboth{#1}{#2}}
\newcommand{\auth}[2]{{#1, #2.}}

\newcommand{\art}[6]{{\sc #1, \rm #2, \it #3 \bf #4 \rm (#5), \mbox{#6}.}}
\newcommand{\artin}[3]{{\sc #1, \rm #2. In #3.}}
\newcommand{\book}[3]{{\sc #1, \it #2, \rm #3.}}
\newcommand{\AND}{{\rm and }}
\RequirePackage{amsthm}
\newtheoremstyle{descriptive}%
  {\topsep}   
  {\topsep}   
  {\rmfamily} 
  {}          
  {\bfseries} 
  {.}         
  { }         
  {}          
\newtheoremstyle{propositional}%
  {\topsep}   
  {\topsep}   
  {\itshape}  
  {}          
  {\bfseries} 
  {.}         
  { }         
  {}          
\newtheoremstyle{remarkstyle}%
  {\topsep}   
  {\topsep}   
  {\rmfamily}  
  {}          
  {\itshape} 
  {.}         
  { }         
  {}          
\theoremstyle{propositional}
\newtheorem{thm}{Theorem}[section]
\newtheorem{lem}[thm]{Lemma}
\newtheorem{prop}[thm]{Proposition}
\newtheorem{cor}[thm]{Corollary}
\theoremstyle{descriptive}
\newtheorem{defi}[thm]{Definition}
\newtheorem{remark}[thm]{Remark}
\newtheorem{example}[thm]{Example}
\makeatletter
\renewenvironment{proof}[1][\proofname]{\par
  \pushQED{\qed}%
  \normalfont
  \trivlist
  \item[\hskip\labelsep
        \itshape
    #1\@addpunct{.}]\ignorespaces
}{%
  \popQED\endtrivlist\@endpefalse
}
\makeatother
\newcommand{\setm}{\setminus}
{\catcode`p =12 \catcode`t =12 \gdef\eeaa#1pt{#1}}      
\def\accentadjtext#1{\setbox0\hbox{$#1$}\kern   
                \expandafter\eeaa\the\fontdimen1\textfont1 \ht0 }
\def\accentadjscript#1{\setbox0\hbox{$#1$}\kern 
                \expandafter\eeaa\the\fontdimen1\scriptfont1 \ht0 }
\def\accentadjscriptscript#1{\setbox0\hbox{$#1$}\kern   
                \expandafter\eeaa\the\fontdimen1\scriptscriptfont1 \ht0 }
\def\accentadjtextback#1{\setbox0\hbox{$#1$}\kern       
                -\expandafter\eeaa\the\fontdimen1\textfont1 \ht0 }
\def\accentadjscriptback#1{\setbox0\hbox{$#1$}\kern     
                -\expandafter\eeaa\the\fontdimen1\scriptfont1 \ht0 }
\def\accentadjscriptscriptback#1{\setbox0\hbox{$#1$}\kern 
                -\expandafter\eeaa\the\fontdimen1\scriptscriptfont1 \ht0 }
\def\itoverline#1{{\mathsurround0pt\mathchoice
        {\rlap{$\accentadjtext{\displaystyle #1}
                \accentadjtext{\vrule height1.593pt}
                \overline{\phantom{\displaystyle #1}
                \accentadjtextback{\displaystyle #1}}$}{#1}}
        {\rlap{$\accentadjtext{\textstyle #1}
                \accentadjtext{\vrule height1.593pt}
                \overline{\phantom{\textstyle #1}
                \accentadjtextback{\textstyle #1}}$}{#1}}
        {\rlap{$\accentadjscript{\scriptstyle #1}
                \accentadjscript{\vrule height1.593pt}
                \overline{\phantom{\scriptstyle #1}
                \accentadjscriptback{\scriptstyle #1}}$}{#1}}
        {\rlap{$\accentadjscriptscript{\scriptscriptstyle #1}
                \accentadjscriptscript{\vrule height1.593pt}
                \overline{\phantom{\scriptscriptstyle #1}
                \accentadjscriptscriptback{\scriptscriptstyle #1}}$}{#1}}}}
\def\vint{\mathop{\mathchoice%
          {\setbox0\hbox{$\displaystyle\intop$}\kern 0.22\wd0%
           \vcenter{\hrule width 0.6\wd0}\kern -0.82\wd0}%
          {\setbox0\hbox{$\textstyle\intop$}\kern 0.2\wd0%
           \vcenter{\hrule width 0.6\wd0}\kern -0.8\wd0}%
          {\setbox0\hbox{$\scriptstyle\intop$}\kern 0.2\wd0%
           \vcenter{\hrule width 0.6\wd0}\kern -0.8\wd0}%
          {\setbox0\hbox{$\scriptscriptstyle\intop$}\kern 0.2\wd0%
           \vcenter{\hrule width 0.6\wd0}\kern -0.8\wd0}}%
          \mathopen{}\int}

\newcommand{\Cp}{{C_p}}
\newcommand{\grad}{\nabla}

\DeclareMathOperator{\Div}{div}
\newcommand{\dvg}{\Div}
\DeclareMathOperator{\cp}{cap}

\newcommand{\capt}{\cp_{p,G_t}}

\newcommand{\bdry}{\partial}
\newcommand{\bdy}{\bdry}
\newcommand{\loc}{_{\rm loc}}

\newcommand{\simge}{\gtrsim}
\newcommand{\simle}{\lesssim}

\newcommand{\limplus}{{\mathchoice{\vcenter{\hbox{$\scriptstyle +$}}}
  {\vcenter{\hbox{$\scriptstyle +$}}}
  {\vcenter{\hbox{$\scriptscriptstyle +$}}}
  {\vcenter{\hbox{\scalebox{0.8}{$\scriptscriptstyle +$}}}}
}}
\newcommand{\limminus}{{\mathchoice{\vcenter{\hbox{$\scriptstyle -$}}}
  {\vcenter{\hbox{$\scriptstyle -$}}}
  {\vcenter{\hbox{$\scriptscriptstyle -$}}}
  {\vcenter{\hbox{\scalebox{0.8}{$\scriptscriptstyle -$}}}}
}}

\newcommand{\A}{{\mathcal A}}
\newcommand{\B}{{\mathcal A}}
\newcommand{\al}{\alpha}

\newcommand{\la}{\lambda}
\newcommand{\ka}{\kappa}

\newcommand{\de}{\delta}

\newcommand{\Om}{\Omega}
\newcommand{\om}{\omega}
\renewcommand{\phi}{\varphi}
\newcommand{\eps}{\varepsilon}
\newcommand{\p}{{$p\mspace{1mu}$}}
\newcommand{\R}{\mathbf{R}}

\newcommand{\ub}{\bar{u}}
\newcommand{\vb}{\bar{v}}
\newcommand{\ut}{\tilde{u}}
\newcommand{\uh}{\hat{u}}
\newcommand{\vh}{\hat{v}}
\newcommand{\vt}{\tilde{v}}
\newcommand{\ft}{\tilde{f}}
\newcommand{\gt}{\tilde{g}}
\newcommand{\fb}{\bar{f}}
\newcommand{\gb}{\bar{g}}
\newcommand{\fh}{\hat{f}}
\newcommand{\wt}{\widetilde{w}}
\newcommand{\Ft}{\widetilde{F}}
\newcommand{\Kt}{\widetilde{K}}
\newcommand{\Khat}{\Kt}   

\newcommand{\Et}{\widetilde{E}}
\newcommand{\Tt}{\widetilde{T}}
\newcommand{\clG}{\itoverline{G}}

\newcommand{\clBprime}{{\,\overline{\!B'}}}
\newcommand{\clBi}{{\,\overline{\!B^i}}}
\newcommand{\clB}{\itoverline{B}}
\newcommand{\clD}{\itoverline{D}}

\newcommand{\Dplus}{D_\limplus}
\newcommand{\Dminus}{D_\limminus}

\newcommand{\Wp}{W^{1,p}}
\newcommand{\Llp}{L^{1,p}}

\newcommand{\Hp}{H^{1,p}}
\newcommand{\Wploc}{W^{1,p}\loc}
\newcommand{\phit}{{\widetilde{\phi}}}
\newcommand{\phib}{{\itoverline{\phi}}}

\newcommand{\eqv}{\ensuremath{\mathchoice{\quad \Longleftrightarrow \quad}{\Leftrightarrow}}
                {\Leftrightarrow}{\Leftrightarrow}}

\def\cprime{{\mathsurround0pt$'$}}
\def\vint{\mathop{\mathchoice%
          {\setbox0\hbox{$\displaystyle\intop$}\kern 0.22\wd0%
           \vcenter{\hrule width 0.6\wd0}\kern -0.82\wd0}%
          {\setbox0\hbox{$\textstyle\intop$}\kern 0.2\wd0%
           \vcenter{\hrule width 0.6\wd0}\kern -0.8\wd0}%
          {\setbox0\hbox{$\scriptstyle\intop$}\kern 0.2\wd0%
           \vcenter{\hrule width 0.6\wd0}\kern -0.8\wd0}%
          {\setbox0\hbox{$\scriptscriptstyle\intop$}\kern 0.2\wd0%
           \vcenter{\hrule width 0.6\wd0}\kern -0.8\wd0}}%
          \mathopen{}\int}
\numberwithin{equation}{section}
\newenvironment{ack}{\medskip{\it Acknowledgement.}}{}
\begin{document}

\authortitle{Jana Bj\"orn  and Abubakar Mwasa}
{Mixed boundary value problem for \p-harmonic functions in an infinite 
cylinder}
\title{Mixed boundary value problem for \\
\p-harmonic functions in an infinite 
cylinder}
\author{
Jana Bj\"orn \\
\it\small Department of Mathematics, Link\"oping University, \\
\it\small SE-581 83 Link\"oping, Sweden\/{\rm ;}
\it \small jana.bjorn@liu.se
\\
\\
Abubakar Mwasa \\
\it\small Department of Mathematics, Link\"oping University, \\
\it\small SE-581 83 Link\"oping, Sweden\/{\rm ;}
\it \small abubakar.mwasa@liu.se
\\ 
\it\small Department of Mathematics, Busitema University, \\
\it\small P.O.Box 236, Tororo, Uganda\/{\rm ;}
\it \small a.mwasa@yahoo.com
}
\date{}

\maketitle

\noindent{\small {\bf Abstract}.
We study a mixed boundary value problem for the \p-Laplace equation 
$\Delta_p u=0$ in an open infinite circular half-cylinder
with prescribed Dirichlet boundary data on a part of the boundary
and zero Neumann boundary data on the rest. 
Existence of weak solutions to the mixed problem
is proved both for Sobolev and for continuous data on the Dirichlet 
part of the boundary.
We also obtain a boundary regularity result for the point at infinity
in terms of a variational capacity adapted to the cylinder.
}

\bigskip

\noindent {\small \emph{Key words and phrases}: Boundary regularity, capacity, 
Dirichlet and Neumann data, existence of weak solutions, mixed boundary value problem, 
\p-Laplace equation, unbounded cylinder, Wiener criterion.
}

\medskip


\section{Introduction}
\label{sect-intr}
When solving the Dirichlet problem for a given partial differential equation, 
in a nonempty open set $\Om\subset \R^n$, one primarily 
seeks a solution $u$ which is constructed from the boundary data 
$f\in C(\bdy\Om)$ so that
\begin{equation}		\label{eq-reg}
\lim_{\Om\ni x\to x_0}u(x)=f(x_0)\quad\text{for } x_0\in\bdy\Om.
\end{equation}
This may or may not be possible for all boundary points.
Therefore, the solution $u$ is often found in a suitable Sobolev space associated 
with the studied equation and the boundary data are only attained in a weak sense.
We say that $x_0\in\bdy\Om$ is \emph{regular} for the considered equation 
if \eqref{eq-reg} holds for all continuous boundary data $f$.
If all the boundary points are regular, the solution attains its continuous boundary 
data in the classical sense.

At irregular boundary points, equality \eqref{eq-reg} may fail even for
continuous boundary data.
The first examples of this phenomenon were given for the Laplace equation $\Delta u=0$ 
in 1911 by Zaremba~\cite{zar} in the punctured ball and in 1912 by 
Lebesgue~\cite{Les} in the complement of the so-called Lebesgue spine.   

Regularity of a boundary point $x_0\in\bdy\Om$ for the Laplace equation $\Delta u=0$ 
can be characterized by the celebrated \emph{Wiener criterion} which was established 
in $1924$ by Wiener~\cite{Wie}. 
With this criterion, one measures the thickness of the complement of $\Om$ 
near $x_0$ in terms of capacities.
Roughly speaking, $x_0$ is regular if the complement is thick enough at $x_0$.

Boundary regularity has been later studied for more general elliptic equations, 
mainly in bounded open sets.
These studies include linear uniformly elliptic equations 
with bounded measurable coefficients in Littman--Stampacchia--Weinberger~\cite{LSW}, 
degenerate linear elliptic equations in Fabes--Jerison--Kenig~\cite{FJK}, 
as well as many nonlinear elliptic equations.
In particular, Maz\cprime ya~\cite{VGM} obtained pointwise estimates 
near a boundary point for weak solutions 
of elliptic quasilinear equations, including the \p-Laplace 
equation~\eqref{eq-p-lap}.
These estimates lead to a sufficient condition for boundary regularity 
for such equations.
Gariepy--Ziemer~\cite{GZ} generalized Maz\cprime ya's result to a larger class of 
elliptic quasilinear equations.

The necessity part of the Wiener criterion for elliptic quasilinear equations 
was for $p>n-1$ proved by Lindqvist--Martio~\cite{LM} and
for all $p>1$ by Kilpel\"ainen--Mal\' y~\cite{TK-JM}.
For weighted elliptic quasilinear equations, the sufficiency part 
was obtained in Heinonen--Kilpel\"ainen--Martio~\cite{HKM}, 
while the necessity condition was established by Mikkonen~\cite{M}.

In this paper, we consider a mixed boundary value problem for the \p-Laplace equation
\begin{equation}   \label{eq-p-lap}
\Delta_pu:=\Div(|\grad u|^{p-2}\grad u)=0, \quad 1<p<\infty, 
\end{equation}
in an open infinite circular half-cylinder  
with zero Neumann boundary data on a part of the boundary and prescribed Dirichlet data 
on the rest of the boundary. 
In Theorem~\ref{thm-ex-sol-Lip}, we prove the existence of weak 
solutions to the mixed boundary value problem 
for \eqref{eq-p-lap} with Sobolev type Dirichlet data. 
For continuous Dirichlet data, we obtain the following result.

\begin{thm}	\label{thm-cont}
Let $G=B'\times(0,\infty)$ be the open infinite circular half-cylinder in $\R^n$, 
where $B'$ is the unit ball in $\R^{n-1}$, and $F$ be an unbounded closed subset 
of $\clG$ containing the base $B'\times\{0\}$ of $G$.
Let $f$ be a continuous function on $F_0:=F\cap\bdy(G\setm F)$ 
such that the limit 
\begin{equation} \label{eq-limf}
f(\infty):= \lim_{F_0\ni x\to\infty}  f(x)
\quad \text{exists and is finite.}
\end{equation}

Then there exists a bounded continuous weak solution $u\in\Wploc(G\setm F)$ 
of the \p-Laplace equation \eqref{eq-p-lap} in $G\setm F$,
with zero Neumann boundary data on $\bdy G\setm F$,
 attained in the weak sense of \eqref{eq-weak-sol-p-Lapl},
and Dirichlet boundary data $f$ on~$F_0$,  attained as the limit
\begin{equation}          \label{eq-lim=f-qe}
\lim_{G\setm F\ni x\to x_0}u(x)=f(x_0)
\end{equation}
for all $x_0\in F_0$, except possibly for a set of Sobolev $C_p$-capacity zero.
Moreover,  the limit $\lim_{G\setm F\ni x\to x_0}u(x)$
exists and is finite for all $x_0\in \bdy G\setm F$.
\end{thm}

Note that the set $F$ need not be a part of the boundary $\bdy G$ and thus, 
equation \eqref{eq-p-lap} can be considered on a more general subset of the
cylinder $G$. 
The zero Neumann condition is, however, prescribed only on a part of
the lateral boundary~$\bdy G$.

We also study boundary regularity of the point at infinity for these solutions.
More precisely, in Theorem~\ref{thm-reg-infty} we show that for continuous 
Dirichlet boundary data $f$ satisfying \eqref{eq-limf},
the solution $u$ satisfies
\[
\lim_{G\setm F\ni x\to \infty}u(x)=f(\infty)
\]
if and only if the Dirichlet part of the boundary is sufficiently large in terms of 
a certain capacity, namely: 
\begin{equation}   \label{eq-Wiener-intro}
\int_1^\infty\cp_{p,G_{t-1}} \bigl(F\cap\bigl(\clBprime 
                    \times [t,2t]\bigr)\bigr)^{1/(p-1)}\,dt=\infty.
\end{equation}
Here the capacity $\cp_{p,G_{t-1}}$ is for compact sets 
\(
K\subset G_{t-1}:= \clBprime \times (t-1,\infty)
\)
defined by 
\[
\cp_{p,G_{t-1}}(K)=\inf_v \int_{G_{t-1}}|\grad v|^p\,dx,
\]
with the infimum taken over all $v\in C_0^\infty(\R^n)$ satisfying $v\ge 1$ on $K$ 
and $v=0$ on $\clG\setm G_{t-1}$.
We also relate $\cp_{p,G_{t-1}}$ to the standard Sobolev 
\p-capacity in~$\R^n$. 
In particular, Lemmas~\ref{lem-cap-Cp} and~\ref{lem-Cp-cap} show that
for $K\subset \clG_t\setm G_{t+1}$, the two capacities are comparable,
but this is not true for general $K\subset G_{t-1}$.

To obtain these results, we 
use the change of variables introduced in Bj\"orn~\cite{JB} to 
transform the infinite half-cylinder $G$ and the \p-Laplace equation \eqref{eq-p-lap} 
into a unit half-ball and a weighted elliptic quasilinear equation 
\begin{equation} \label{eq-ell-eq}
\Div\A(\xi,\grad u(\xi))=0,
\end{equation}
respectively. 
In order to use the theory of Dirichlet problems,
developed in Heinonen--Kilpel\"ainen--Martio~\cite{HKM} for the 
equation~\eqref{eq-ell-eq}, the Neumann data are removed by reflecting 
the unit half-ball and the equation \eqref{eq-ell-eq} to the whole unit ball. 
We then use the Wiener criterion for such equations, together with tools 
from~\cite{HKM}, to determine the regularity of the point at infinity 
and to prove the existence of continuous weak solutions to the mixed boundary value problem 
for~\eqref{eq-p-lap}.

Compared to the Dirichlet problem in bounded domains, there are relatively 
few studies of boundary value problems with respect to unbounded domains 
and with mixed boundary data. 
Early work on mixed boundary value problems was due to Zaremba~\cite{ZarProb}
and such problems are therefore sometimes called \emph{Zaremba problems}.
Kerimov--Maz\cprime ya--Novruzov~\cite{KMV} characterized regularity of the point 
at infinity for the Zaremba problem for the Laplace equation $\Delta u=0$ 
in an infinite half-cylinder. 
Bj\"orn~\cite{JB} studied a similar problem for certain linear 
weighted elliptic equations.
Our results partially extend the ones in \cite{JB} and \cite{KMV} to 
the \p-Laplace equation~\eqref{eq-p-lap}, even though the necessary
and sufficient conditions obtained therein  are formulated differently.

The organization of the paper is as follows.
In Section~\ref{sect-pre-not} we introduce the notation and give the
definition of weak solutions.
Section~\ref{sect-cyl-ball} is devoted to transforming the infinite 
half-cylinder together with the \p-Laplace equation \eqref{eq-p-lap} 
into a unit half-ball with the weighted elliptic quasilinear 
equation~\eqref{eq-ell-eq}.
In Section~\ref{sect-pro-ope-B}, we state and prove some properties 
of the obtained operator $\Div\A(\xi,\grad u(\xi))$,  
such as ellipticity and monotonicity, needed to apply the results from 
Heinonen--Kilpel\"ainen--Martio~\cite{HKM}.
   
In Section~\ref{sect-rem-Neu-data}, the Neumann boundary data are removed 
by means of a reflection and the
mixed boundary value problem is turned into a Dirichlet problem.
This makes it is possible to use the tools developed for 
weighted elliptic quasilinear equations in \cite{HKM}.
Sections~\ref{sect-pro-ope-B} and~\ref{sect-rem-Neu-data} also contain comparisons 
of appropriate function spaces on the half-cylinder and those on the ball.
In Section~\ref{sec-existence}, we prove the existence of continuous weak solutions 
to the mixed boundary value problem for \eqref{eq-p-lap}. 
Section~\ref{sect-cap} is devoted to comparing two variational capacities: 
one associated with the weighted Sobolev spaces on the unit ball 
and the other defined on the half-cylinder. 
These are crucial for studying the boundary regularity at infinity in Section~\ref{sect-bdry-reg}.

\begin{ack}
J.~B. was partially supported by the Swedish Research Council 
grants 621-2014-3974 and~2018-04106.
A.~M. was supported by the SIDA (Swedish International Development Cooperation Agency)
project 316-2014  ``Capacity building in Mathematics and its applications'' 
under the SIDA bilateral program with the Makerere University 2015--2020,
contribution No.\ 51180060.
\end{ack}


\section{Notation and formulation of the mixed problem}
\label{sect-pre-not}

Throughout the paper, we represent points in the 
$n$-dimensional Euclidean space $\R^n=\R^{n-1}\times \R$, $n\geq 2$, as 
$x=(x',x_n)=(x_1,\cdots,x_{n-1}, x_n)$. 
We shall consider the open infinite circular half-cylinder 
\[
G=B'\times (0,\infty),
\]
where $B'=\{x'\in\R^{n-1}:|x'|<1\}$ is the unit ball in $\R^{n-1}$.

Let $F$ be a closed subset of $\itoverline{G}$. 
Assume also that $F$ contains the base $B'\times \{0\}$ of $G$.
Let $1<p<\infty$ be fixed.
We shall consider a \emph{mixed boundary value problem} for the 
\p-Laplace equation $\Delta_p u=0$  in $G\setm F$ with Dirichlet boundary data 
\[
u=f\quad\text{on }F\cap \bdy(G\setm F)=:F_0
\] 
and zero Neumann boundary data 
\[
\frac{\partial u}{\partial n}=0\quad\text{on }\partial G\setm F,
\] 
where $n$ is the outer normal of $G$.

Note that $F$ is not necessarily a subset of $\bdy G$, which makes it possible 
to consider more general domains contained in $G$. 
If $\bdy G\subset F$, then the mixed boundary value problem reduces to a purely 
Dirichlet problem on such domains contained in $G$.

The \p-Laplace equation~\eqref{eq-p-lap} and the Neumann 
condition will be considered in the weak sense as follows:

\begin{defi}    \label{defi-weak-p-lap}
A function $u\in \Wp\loc(G\setm F)$ is a \emph{weak solution} 
of the mixed boundary value problem for the \p-Laplace 
equation $\Delta_pu=0$ in $G\setm F$ with \emph{zero Neumann boundary data} 
on $\partial G\setm F$ if the integral identity 
\begin{equation} \label{eq-weak-sol-p-Lapl}
\int_{G\setm F} |\nabla u|^{p-2}\nabla u \cdot \nabla\phi\, dx=0
\quad \text{holds for all }    \phi \in C_0^\infty(\clG\setm F),
\end{equation}
where $\cdot$ denotes the scalar product in $\R^n$ and 
\begin{equation}   \label{eq-def-C0-infty}
C_0^\infty(\clG\setm F) := 
    \{v|_{\clG}: v\in C_0^\infty(\R^n\setm F)\}.
\end{equation}
Recall that for an open set $\Om\subset\R^n$, the space $C_0^{\infty}(\Om)$ 
consists of all infinitely many times continuously differentiable functions with 
compact support in  $\Om$, extended by zero outside $\Om$ if needed.
\end{defi}

In \eqref{eq-weak-sol-p-Lapl} it is implicitly assumed that the 
integral exists for all test functions
$\phi \in C_0^{\infty}(\clG\setm F)$.
This need not be the case for a general $u\in \Wp\loc(G\setm F)$.

\section{Transforming the half-cylinder into a half-ball}  
\label{sect-cyl-ball}

In this section, we shall see that the \p-Laplace operator in the open 
infinite circular half-cylinder $G$
corresponds to a weighted quasilinear elliptic operator on the unit half-ball.
The following change of variables was 
introduced in Bj\"orn~\cite[Section~3]{JB}.

Let $\ka>0$ be a fixed constant and define 
\begin{equation}   \label{eq-def-xi-from-x}
\xi'= \frac{2e^{-\ka x_n} x'}{1+|x'|^2} \quad \text{and} \quad
\xi_n=\frac{e^{-\ka x_n}(1-|x'|^2)}{1+|x'|^2},
\end{equation}
where we adopt the notation 
$\xi=(\xi',\xi_n)=(\xi_1,\ldots,\xi_{n-1},\xi_n)\in \R^n$, 
similar to $x=(x',x_n)\in \R^n$.
The mapping $x\mapsto\xi=T(x)$ is defined on $\R^n$ with values in
\[
T(\R^n) = \R^n\setm\{(\xi',\xi_n)\in \R^n: \xi'=0\text{ and }\xi_n\leq 0\}.
\]
We will mainly consider $T$ on $G$ and its closure $\clG$.
It is easily verified that
\begin{align*}
T(G) &= \{\xi\in\R^n: |\xi|<1\text{ and } \xi_n>0\}, \\ 
T(\clG) &= \{\xi\in\R^n: 0<|\xi|\le1\text{ and } \xi_n\ge0\}
\end{align*}
are the open and the closed upper unit half-balls, respectively, with the
origin $\xi=0$ removed. 
Note that 
\[
|\xi|=|T(x)|=e^{-\ka x_n} \to 0 \quad\text{as } x_n \to \infty,
\]
so the point at infinity for the half-cylinder $G$ corresponds to 
the origin $\xi=0$.
Throughout the paper we will use $x$ for points in $\clG$,
while $\xi$ will be used for points in the target space of $T$.

A direct calculation shows that the inverse mapping $T^{-1}$ of $T$
is given by:
\begin{equation}   \label{eq-x'-x-n-from-xi}
x'= \frac{\xi'}{|\xi|+\xi_n} \quad \text{and} \quad
x_n=-\frac{1}{\ka}\log|\xi|.
\end{equation}
The following lemma is then easily proved by induction.

\begin{lem}   \label{lem-deriv}
Let $\al=(\al_1,\al_2,\ldots,\al_n)$ be a multiindex of order $m\ge1$,
that is, $\al_j$ are nonnegative integers, $j=1,2,\ldots,n$,
and $m=\al_1+\ldots +\al_n$.
The partial derivatives of $T$ can then be written in the form
\[
\bdy^\al \xi_k(x) := \bdy^{\al_1}_{x_1} \cdots \bdy^{\al_n}_{x_n} \xi_k(x)
= \frac{\ka^{\al_n} e^{-\ka x_n} P_{k,\al}(x')}{(1+|x'|^2)^{m+1}},
\quad k=1,2,\ldots, n,
\]
where $P_{k,\al}(x')$ are polynomials in $x_1, \ldots, x_{n-1}$ with integer
coefficients.

Conversely, the partial derivatives of the inverse mapping are
\begin{align*}
\bdy^\al x_k(\xi)
&= \sum_{\substack{ 0\le j+l\le 2m+1 \\ j,l \ge0}{}} 
     \frac{P_{j,k,l,\al}(\xi)}{|\xi|^{j} (|\xi|+\xi_n)^{l}},
\quad k=1,2,\ldots, n-1, \\
\bdy^\al x_n(\xi) &= \frac{P_\al(\xi)}{\ka|\xi|^{2m}},
\end{align*}
where $P_{j,k,\al}(\xi)$ and $P_\al(\xi)$ are polynomials in $\xi_1, \ldots, \xi_{n}$ 
with integer coefficients.
\end{lem}

Note that 
\begin{equation}   \label{eq-xi+xi-n}
|\xi|+\xi_n = \frac{2e^{-\ka x_n}}{1+|x'|^2}
\end{equation}
is positive and bounded away from $0$ as long as $x$ stays within a bounded set
in $\R^{n}$, or equivalently, as long as $\xi$ stays away from $0$ and from 
the negative $\xi_n$-axis.
In particular, $|\xi|(|\xi|+\xi_n)>0$ in $T(\R^n)$ and hence
$T$ is a smooth diffeomorphism between $\R^n$ and $T(\R^n)$.

Our next step is to see how the \p-Laplace equation \eqref{eq-p-lap}
transforms under the diffeomorphism $T$.
For notational purposes, we regard the differential 
\[
dT(x):h\longmapsto dT(x)h
\]
of $T$ as the left-multiplication of the
column vector $h\in\R^n$ by the Jacobian matrix of partial derivatives
\[
dT(x):=\begin{pmatrix}
\frac{\partial\xi_1}{\partial x_1}&\cdots & \frac{\partial\xi_1}{\partial x_n}\\[1mm]
\frac{\partial\xi_2}{\partial x_1}&\cdots & \frac{\partial\xi_2}{\partial x_n}\\
\vdots & \ddots &\vdots\\
\frac{\partial\xi_n}{\partial x_1}&\cdots & \frac{\partial\xi_n}{\partial x_n}\\
\end{pmatrix}.
\]
With this matrix convention, the chain rule for $u$ and 
$\ut=u\circ T^{-1}$ can be written as 
\begin{equation}  \label{eq-chain-rule}
\grad u(x)=dT^*(x)\grad \ut(\xi), 
\quad \text{where }\xi=T(x), 
\end{equation}
$dT^*(x)$ is the transpose of the matrix $dT(x)$ and
the distributional gradients $\grad u(x)$ and $\grad \ut(\xi)$ 
are seen as column vectors in $\R^n$. 

Formula \eqref{eq-chain-rule} clearly holds when $u$ and $\ut$
are smooth, while for functions in $L^1\loc$ with distributional gradients  
in $L^1\loc$ it is obtained by mollification and holds a.e.,
see for example Ziemer~\cite[Theorem~2.2.2 and Section~1.6]{Zie}
or H\"ormander~\cite[Section~6.1]{horm}.

We shall substitute the chain rule (\ref{eq-chain-rule}) into equation 
\eqref{eq-weak-sol-p-Lapl} to obtain the corresponding integral identity on the 
unit half-ball $T(G)$.

\begin{lem}              \label{lem-int-id-ball}
Let $u,\phi\in \Wp(U)$ for some open $U\subset G$ and set 
$\ut= u\circ T^{-1}$ and $\phit= \phi\circ T^{-1}$.
Then for any measurable $A\subset U$,
\begin{equation}  \label{eq-p-Lapl-to-B}
\int_{A} |\nabla u|^{p-2}\nabla u\cdot\nabla \phi\, dx
=\int_{T(A)} \B(\xi,\grad\ut)\cdot\grad \phit \,d\xi,
\end{equation}
where $\B$ is for $\xi=T(x)\in T(G)$ and $q\in\R^n$ defined by
\begin{equation}   \label{eq-def-B}
\B(\xi,q)= |dT^*(x)q|^{p-2}|J_T(x)|^{-1} dT(x)dT^*(x)q. 
\end{equation}
\end{lem}
Here, $J_T(x) = \det(dT(x))$ denotes the Jacobian of $T$ at $x$.

\begin{proof}
It will be convenient to use the above matrix notation. 
Rewrite the scalar product on the left-hand side 
of \eqref{eq-p-Lapl-to-B} as
\[
|\grad u|^{p-2} \grad u \cdot\grad \phi 
     =|\grad u|^{p-2}(\grad u)^* \grad \phi
\]
and apply the chain rule \eqref{eq-chain-rule}.
Using the change of variables $\xi=T(x)$, we obtain 
\begin{align*}
&\int_{A}|\grad u|^{p-2}(\grad u)^* \grad \phi \,dx \\
& \quad \quad \quad
= \int_{T(A)}|dT^*(x)\grad\ut|^{p-2} \bigl( dT^*(x)\grad\ut \bigr)^*
        \bigl( dT^*(x)\grad\phit \bigr) |J_T(x)|^{-1} \,d\xi \\
&\quad \quad \quad
=\int_{T(A)}|dT^*(x)\grad\ut|^{p-2}|J_T(x)|^{-1}  
       \bigl( dT(x)dT^*(x)\grad\ut \bigr)^*  \grad\phit \,d\xi\\
&\quad \quad \quad
=\int_{T(A)}\B(\xi,\grad\ut)\cdot\grad\phit\,d\xi. 
\end{align*}
Note that by the assumptions on $u$ and $\phi$, all the integrals
are finite.
\end{proof}

In view of the integral identity \eqref{eq-weak-sol-p-Lapl}, 
Lemma~\ref{lem-int-id-ball} indicates
that the \p-Laplace equation (\ref{eq-p-lap}) 
on $G\setm F$ will be transformed by $T$ into the equation
\begin{equation}   \label{eq-div-B}
\Div \B(\xi,\grad \ut(\xi))=0  \quad \text{on } T(G\setm F),
\end{equation}
with a proper interpretation of the function spaces and
the zero Neumann condition, which will be made precise later,
see Proposition~\ref{prop-Wp-Hp-loc-divB}, Theorem~\ref{thm-pLapl-after-B}
and Section~\ref{sec-existence}.

In the next section, we will study the operator \eqref{eq-div-B} in more detail.
For this, we will use the following  geometric lemma.
Its proof is rather straightforward, but requires good control of all the involved expressions.
We provide it for the reader's convenience.

Throughout the paper, unless otherwise stated, $C$ will denote any 
positive constant whose real value is not important and need 
not be the same at each point of use. 
It can even vary within a line. 
By $a\simle b$ we mean that there exists a nonnegative 
constant $C$, independent of $a$ and $b$, such that $a\le Cb$.
We also write $a\simeq b$ if $a\simle b\simle a$.

\begin{lem}   \label{lem-JB}
For all $x,y\in \R^n$, it holds that  
\begin{equation} \label{eq-Tx-Ty}
\frac{e^{-\ka \max\{x_n,y_n\}} |x-y|}{(1+|y'|^2)(\tfrac12+|x'|)+1/\ka} 
   \le |T(x)-T(y)|\le (5+2\ka) e^{-\ka \min\{x_n,y_n\}}|x-y|.
\end{equation}
In particular, if $|x'|\le M$ and  $q\in\R^n$ then 
\[
|dT^*(x)q|\simeq |dT(x)q|\simeq e^{-\ka x_n}|q|
\quad \text{and}\quad 
|J_{T}(x)|\simeq e^{-\ka nx_n},
\]
where the comparison constants in $\simeq$ depend on $\ka$ and $M$, but
are independent of $x$ and $q$.
 \end{lem} 

\begin{proof}
Let $\xi=T(x)$ and $\eta=T(y)$. 
We can assume that $|y'|\le|x'|$.
By \eqref{eq-def-xi-from-x} and the triangle inequality,
\[
|\xi'-\eta'| \le \frac{2e^{-\ka x_n}|x'-y'|}{1+|x'|^2} 
    + |y'|e^{-\ka x_n} \biggl| \frac{2}{1+|x'|^2} - \frac{2}{1+|y'|^2} \biggr|
    + \frac{2|y'| \,|e^{-\ka x_n}-e^{-\ka y_n}|}{1+|y'|^2}  
\]
and
\[
|\xi_n-\eta_n| \le e^{-\ka x_n} 
    \biggl| \frac{1-|x'|^2}{1+|x'|^2} - \frac{1-|y'|^2}{1+|y'|^2} \biggr|
   + \frac{\bigl| 1-|y'|^2 \bigr|\, |e^{-\ka x_n} - e^{-\ka y_n}|}{1+|y'|^2}.
\]
In the above two estimates, we have $\bigl|1-|y'|^2\bigr|/(1+|y'|^2)\le1$,
\[
\biggl| \frac{1-|x'|^2}{1+|x'|^2} - \frac{1-|y'|^2}{1+|y'|^2} \biggr|
= \biggl| \frac{2}{1+|x'|^2} - \frac{2}{1+|y'|^2} \biggr|
\le \frac{2|x'-y'| (|x'|+|y'|)}{(1+|x'|^2)(1+|y'|^2)}
\]
and
\[
\frac{2|y'| (|x'|+|y'|)}{(1+|x'|^2)(1+|y'|^2)} 
\le \frac{2|x'|}{1+|x'|^2} \, \frac{2|y'|}{1+|y'|^2}\le 1.
\]
Since the mean-value theorem shows that
\[
|e^{-\ka x_n}-e^{-\ka y_n}| \le \ka e^{-\ka \min\{x_n,y_n\}} |x_n-y_n|,
\]
we conclude that
\[
|T(x)-T(y)| \le |\xi'-\eta'| + |\xi_n-\eta_n|
   \le e^{-\ka \min\{x_n,y_n\}} (2+1+\ka+2+\ka) |x-y|,
\]
which gives the second inequality in \eqref{eq-Tx-Ty}.
Conversely, \eqref{eq-x'-x-n-from-xi} and the triangle inequality yield
\[
|x'-y'| \le \biggl| \frac{\xi'}{|\xi|+\xi_n} - 
    \frac{\xi'}{|\eta|+\eta_n} \biggr| + \frac{|\xi'-\eta'|}{|\eta|+\eta_n} 
\le |\xi'| \frac{|\xi-\eta| + |\xi_n-\eta_n|}{(|\xi|+\xi_n)(|\eta|+\eta_n)} 
       + \frac{|\xi'-\eta'|}{|\eta|+\eta_n},
\]
where
\[
\frac{|\xi'|}{|\xi|+\xi_n} = |x'| \quad \text{and} \quad
\frac{1}{|\eta|+\eta_n} = \frac{1+|y'|^2}{2e^{-\ka y_n}},
\]
because of \eqref{eq-x'-x-n-from-xi} and \eqref{eq-xi+xi-n}.
Since also
\[
|x_n-y_n| = \frac{1}{\ka} \bigl| \log|\xi| - \log|\eta| \bigr|
    \le \frac{ \bigl| |\xi| - |\eta| \bigr|}{\ka \min\{|\xi|,|\eta|\}}
    \le \frac{ |\xi-\eta|}{\ka e^{-\ka\max\{x_n,y_n\}}},
\]
where the first inequality follows from the mean-value theorem, 
we conclude that 
\[
|x-y| \le |x'-y'| + |x_n-y_n| \le  
   \frac{(1+|y'|^2) (\tfrac12+|x'|) + 1/\ka}{e^{-\ka \max\{x_n,y_n\}}} |\xi-\eta|,
\]
which proves the first inequality in \eqref{eq-Tx-Ty}.

The estimates $|dT(x)q|\simeq e^{-\ka x_n}|q|$ and $|J_{T}(x)|\simeq e^{-\ka nx_n}$
follow directly from \eqref{eq-Tx-Ty} and the definition of the differential
$dT(x)$.
Hence also, 
\[
|dT^*(x)q|^2 = q^* dT(x) dT^*(x)q \le |q|\, |dT(x) dT^*(x)q| 
  \simeq e^{-\ka x_n} |q| \, |dT(x)^*q|
\]
and
\[
|q|^2 = q^*q = (dT(x)^{-1}q)^*  dT^*(x)q \le |dT(x)^{-1}q|\, |dT^*(x)q| 
  \simeq e^{\ka x_n} |q| \,|dT^*(x)q|.
\]
Dividing by $|dT^*(x)q|$ and $|q|$, respectively, finishes the proof.
\end{proof}

\section{Properties of the operator $\Div \B(\xi,\grad \ut)$}
\label{sect-pro-ope-B}

We shall now study some properties of the operator 
$\Div \B(\xi,\grad \ut)$ on $T(G)$.
It will turn out to be degenerate elliptic 
with a degeneracy given by the weight function
\[
\wt(\xi)=|\xi|^{p-n}, \quad \xi\in\R^n\setm\{0\},
\]
and $\wt(0)=0$.
This will make it possible to treat equation \eqref{eq-div-B}
using methods from Heinonen--Kilpel\"ainen--Martio~\cite{HKM}.

\begin{thm}    \label{thm-B-ellipt} 
The mapping $\B:T(G)\times \R^n\to \R^n$, defined by \eqref{eq-def-B}, 
satisfies the ellipticity conditions
\begin{align*}
\B(\xi,q)\cdot q \simeq \wt(\xi)|q|^p \quad 
&\text{for all } q\in\R^n\text{ and }  \xi\in T(G),  \\ 
|\B(\xi,q)| \simeq \wt(\xi) |q|^{p-1} \quad 
&\text{for all } q\in\R^n \text{ and } \xi\in T(G), 
\end{align*}
where the comparison constants are independent of $\xi$ and $q$.
\end{thm}

\begin{proof}[Proof of Theorem~\ref{thm-B-ellipt}]
Using the above matrix notation and \eqref{eq-def-B}, 
we have for all $q\in\R^n$ and $\xi=Tx\in T(G)$,
\begin{align}
\B(\xi,q)\cdot q &=|dT^*(x)q|^{p-2} |J_T(x)|^{-1} 
          \bigl( dT(x)dT^*(x)q \bigr)^* q 
\nonumber \\
&=|dT^*(x)q|^{p-2} |J_T(x)|^{-1} |dT^*(x)q|^2.  \label{eq-B-times-q}
\end{align}
Lemma~\ref{lem-JB} then gives
\[
\B(\xi,q)\cdot q \simeq e^{-\ka(p-n) x_n}|q|^p= |\xi|^{p-n}|q|^p, 
\]
which concludes the proof of the first statement.
For the second statement, note that by Lemma~\ref{lem-JB}, we have
\[
|dT(x)dT^*(x)q| \simeq e^{-\ka x_n}|dT^*(x)q|\simeq e^{-2\ka x_n}|q|.
\]
Thus by Lemma~\ref{lem-JB} together with \eqref{eq-def-B}, we get
\begin{align*}
|\B(\xi,q)| &= |dT^*(x)q|^{p-2}|J_T(x)|^{-1}|dT(x)dT^*(x)q|\\
&\simeq  e^{-\ka(p-n)x_n}|q|^{p-1} = |\xi|^{p-n}|q|^{p-1}.\qedhere
\end{align*}
\end{proof}

\begin{thm}    \label{thm-B-monot} 
The mapping $\B:T(G)\times \R^n\to \R^n$, defined by \eqref{eq-def-B}, 
satisfies for all $\xi\in T(G)$ and $q_1,q_2\in\R^n$
the  monotonicity condition
\begin{equation}  \label{eq-B-monot}
(\B(\xi,q_1)-\B(\xi,q_2)) \cdot (q_1-q_2) \ge 0, \quad 
\end{equation}
with equality if and only if $q_1=q_2$.
\end{thm}

\begin{proof}
Expand the left-hand side of \eqref{eq-B-monot} as
\begin{equation}  \label{e6}
(\B(\xi,q_1)-\B(\xi,q_2))\cdot(q_1-q_2) =: A_1-A_2 \ge A_1 -|A_2|, 
\end{equation}
where 
\begin{align*}
A_1 &=\B(\xi,q_1)\cdot q_1+ \B(\xi,q_2)\cdot q_2, \\
A_2 &=\B(\xi,q_1)\cdot q_2+\B(\xi,q_2)\cdot q_1. 
\end{align*}
Using \eqref{eq-B-times-q}, we have
\begin{equation*}  
A_1  =a \bigl( |dT^*(x)q_1|^p+|dT^*(x)q_2|^p \bigr),
\end{equation*} 
where $a=|J_T(x)|^{-1}$.
Estimating the first term of $A_2$ using \eqref{eq-def-B}, 
together with the Cauchy--Schwarz and Young inequalities, yields
\begin{align*}
|\B(\xi,q_1)\cdot q_2|&= a|dT^*(x)q_1|^{p-2}|q^*_2dT(x)dT^*(x)q_1| \\
&\leq  a|dT^*(x)q_1|^{p-1}|dT^*(x)q_2|\\  
&\leq a\biggl(\frac{p-1}{p}|dT^*(x)q_1|^{p}+\frac{1}{p}|dT^*(x)q_2|^p\biggr).
\end{align*}
Similarly, with the roles of $q_1$ and $q_2$ interchanged, 
the second term in $A_2$ is estimated as 
\begin{equation*}  
|\B(\xi,q_2)\cdot q_1| 
\leq a\biggl(\frac{p-1}{p}|dT^*(x)q_2|^{p}+\frac{1}{p}|dT^*(x)q_1|^p\biggr).
\end{equation*}
Substituting the last three estimates back into (\ref{e6}) reveals that 
\begin{equation*}   
(\B(\xi,q_1)-\B(\xi,q_2))\cdot(q_1-q_2)\geq 0
\end{equation*}
for all $q_1,q_2\in \R^n.$
We notice that the left-hand side is zero if and only if equality holds 
both in the Cauchy--Schwarz and Young inequalities, from which it 
is easily concluded that this requires $dT^*(x)q_1=dT^*(x)q_2$ and thus
$q_1=q_2$, since $dT^*(x)$ is invertible.
\end{proof}

Theorems~\ref{thm-B-ellipt} and~\ref{thm-B-monot}
show that the ellipticity and monotonicity assumptions
(3.4)--(3.6) in Heinonen--Kil\-pe\-l\"ai\-nen--Martio~\cite{HKM} are satisfied
for $\B$ with the weight $\wt(\xi)=|\xi|^{p-n}$.
Moreover, $\B$ is clearly measurable in $\xi$ and continuous in $q$, so 
(3.3) in~\cite{HKM} holds as well. 
The homogeneity condition~(3.7) in~\cite{HKM} is also obviously satisfied.

The following lemma is well-known, cf.\ 
Heinonen--Kilpel\"ainen--Martio~\cite[p.\ 298]{HKM}.
For the reader's convenience, we include the short proof.
Here and in the rest of the paper, we let 
\[
B_r = B(0,r)= \{\xi \in \R^n: |\xi|<r\}
\] 
denote the open ball centred at the origin and with radius $r>0$.

\begin{lem}\label{lem-w-Ap}
Let $r>0$ and $\al\in\R$. Then
\begin{equation}   \label{eq-int-wt-al}
\int_{B_r} \wt(\xi)^\al \,d\xi = 
\begin{cases}
C_{\al,p,n}r^{n+\al(p-n)} &\text{if } \al(n-p)<n, \\
\infty &\text{otherwise,} \end{cases}
\end{equation}
where 
\[
C_{\al,p,n}= \frac{n\om_n}{n+\al(p-n)}
\]
and $\om_n$ is the Lebesgue measure of the unit ball in $\R^n$.

Moreover, $\wt$ belongs to the Muckenhoupt class $A_p$,
that is, for all balls $B\subset\R^n$, 
\begin{equation}  \label{eq-Ap-for-wt}
\biggl(\int_B \wt(\xi)\,d\xi \biggr) 
   \biggl( \int_B \wt(\xi)^{1/(1-p)}\,d\xi \biggr)^{p-1}  \simle |B|^p,
\end{equation}
where $|B|$ stands for the Lebesgue measure of $B$.
\end{lem}

\begin{proof}
Estimate \eqref{eq-int-wt-al} is easily obtained by direct calculation
using spherical coordinates.
To prove \eqref{eq-Ap-for-wt}, we let $B=B(\zeta,r)$ be a ball
and consider two cases: 

If $r<\tfrac12|\zeta|$, then $\wt(\xi) \simeq \wt(\zeta)$ 
for all $\xi\in B$ 
and hence the left-hand side in \eqref{eq-Ap-for-wt} is comparable to
\[
(\wt(\zeta) |B|) \bigl( \wt(\zeta)^{1/(1-p)}|B| \bigr)^{p-1} = |B|^p.
\]
On the other hand, if $r\ge\tfrac12|\zeta|$, then $B\subset B_{3r}$
and hence, by the first part of the lemma with $\al=1$ and $\al=1/(1-p)$,
\begin{align*}
\biggl( \int_B \wt(\xi)\,d\xi \biggr) 
   \biggl( \int_B \wt(\xi)^{1/(1-p)}\,d\xi \biggr)^{p-1} 
&\le \biggl(\int_{B_{3r}} \wt(\xi)\,d\xi \biggr) 
   \biggl( \int_{B_{3r}} \wt(\xi)^{1/(1-p)}\,d\xi \biggr)^{p-1} \\
   &\simeq (3r)^p \bigl((3r)^{n+(p-n)/(1-p)}\bigr)^{p-1} \\
   &\simeq r^{np}. 
\end{align*}
Note that $\al(n-p)<n$ for both choices of $\al$.\qedhere
\end{proof}

Weights from the Muckenhoupt class $A_p$ are known to be \p-admissible,
i.e.\ the measure $d\mu(\xi)=\wt(\xi)\,d\xi$ is doubling and supports
a \p-Poincar\'e inequality on $\R^n$, see 
Heinonen--Kil\-pe\-l\"ai\-nen--Martio~\cite[Chapters~15 and~20]{HKM}.
Such measures are suitable for the theory of
Sobolev spaces and partial differential equations, as developed in \cite{HKM}.

\begin{defi}
For an open set $\Om\subset\R^n$, the weighted Sobolev space 
$\Hp_0(\Om,\wt)$ is the completion of $C_0^\infty(\Om)$ in the norm
\[
\|u\|_{\Hp(\Om,\wt)} =\biggl(\int_{\Om} \bigl( |u(\xi)|^p
           +|\nabla u(\xi)|^p \bigr) \wt(\xi)\,d\xi \biggr)^{1/p}.
\]
Similarly, $H^{1,p}(\Om,\wt)$ is the completion of the set
\[
\{\phi\in C^{\infty}(\Om): \|\phi\|_{\Hp(\Om,\wt)} <\infty\}
\] 
in the $\Hp(\Om,\wt)$-norm.
\end{defi}

In other words, a function $u$ belongs to $\Hp(\Om,\wt)$ if and only if 
$u \in L^p(\Om,\wt)$ and there is a vector-valued function $v$
such that for some sequence of smooth functions $\phi_k\in C^{\infty}(\Om)$ 
with $\|\phi_k\|_{\Hp(\Om,\wt)}<\infty$, we have
\[
\int_{\Om} |\phi_k-u|^p \wt\,d\xi \to 0
\quad \text{and} \quad 
\int_{\Om} |\grad \phi_k - v|^p \wt\,d\xi \to 0,
\quad \text{as } k\to\infty.
\]

Since $\wt^{1/(1-p)}\in L^p\loc(\R^n,dx)$, we know from \cite[Section~1.9]{HKM}
that $v= \grad u$ is the distributional gradient of $u$.
Moreover, by Kilpel\"ainen~\cite{TK94}, $u\in \Hp(\Om,\wt)$ 
if and only if both $u$ and its distributional gradient $\grad u$ belong
to $L^p(\Om,\wt)$.
For the unweighted Sobolev space with $\wt\equiv1$, 
we use the notation $\Wp(\Om)$.

\begin{defi}
Following \cite[Chapter~3]{HKM}, we say that a function $u\in \Hp\loc(\Om,\wt)$ 
is a \emph{weak solution} of the equation $\Div\B(\xi,\grad u(\xi))=0$
in $\Om$ if for all test functions $\phi\in C_0^\infty(\Om)$, 
\[ 
\int_{\Omega}\B(\xi,\grad u(\xi)) \cdot \grad \phi(\xi)\, d\xi=0.
\] 
\end {defi}

We can now make more precise the statement that the \p-Laplace equation
on $G\setm F$ transforms into the equation \eqref{eq-div-B}.
First, we formulate the following simple consequence of the estimates in
Lemma~\ref{lem-JB},
which will also be useful later when dealing with function spaces 
on $G$ and $T(G)$, and when comparing capacities.

\begin{lem}         \label{lem-int-G-T}
Assume that $u\in L^1\loc(U)$ with the distributional
gradient $\grad u \in L^1\loc(U)$ for some open set 
$U\subset B'(0,R)\times\R$ and let $\ut=u\circ T^{-1}$.  
Then for any measurable set $A\subset U$,
\begin{align*}   
\int_{A}|\grad u|^p\,dx 
  &\simeq \int_{T(A)}|\grad \ut|^p \wt(\xi)\,d\xi, \\
\int_{A}|u|^p e^{-p\ka x_n}\,dx
  &\simeq \int_{T(A)}|\ut|^p \wt(\xi)\,d\xi, \nonumber 
\end{align*}
with comparison constants depending on $R$ but independent of $A$ and $u$.
\end{lem}

Note that, in general, the above integrals can be infinite, but then
they are infinite simultaneously.

\begin{proof}
As in Lemma~\ref{lem-int-id-ball}, we use the change of variables $\xi=T(x)$.
The chain rule \eqref{eq-chain-rule}, together with
Lemma~\ref{lem-JB} and $e^{-\ka x_n}=|\xi|$, implies that 
\[
\int_A |\grad u|^p \,dx 
=\int_{T(A)} |dT^*(x)\grad u|^p |J_T(x)|^{-1}\,d\xi
\simeq \int_{T(A)}|\grad \ut|^p|\xi|^{p-n}\,d\xi
\]
and similarly,
\[
\int_A |u|^p e^{-p\ka x_n}\,dx =\int_{T(A)}|\ut|^p |\xi|^p |J_T(x)|^{-1}\,d\xi
\simeq \int_{T(A)} |\ut|^p|\xi|^{p-n}\,d\xi.\qedhere
\]
\end{proof}

\begin{prop}  \label{prop-Wp-Hp-loc-divB}
A function $u\in \Wp\loc(G\setm F)$ is a weak solution 
of the \p-Laplace equation $\Delta_p u=0$  
in $G\setm F$ if and only if $\ut=u\circ T^{-1}$ is a
weak solution of the equation
$\Div\B(\xi,\grad \ut(\xi))=0$ in $T(G\setm F)$.
\end{prop}

\begin{proof}
Using Lemma~\ref{lem-int-G-T}, we conclude that $u\in \Wp\loc(G\setm F)$
if and only if
$\ut\in\Hp\loc(T(G\setm F),\wt)$, since $e^{-p\ka x_n} \simeq 1$ for every
compact subset of $G\setm F$.
We need to show that $u$ satisfies the integral identity in
\eqref{eq-weak-sol-p-Lapl}
for all test functions $\phi \in C_0^\infty(G\setm F)$ 
if and only if $\ut$ satisfies 
\begin{equation}   \label{eq-div-B-xi-T(G)} 
\int_{T(G\setm F)}\B(\xi,\grad\ut(\xi)) \cdot \grad \phit \,d\xi =0
\end{equation}
for all test functions $\phit\in C_0^\infty(T(G\setm F))$. 
Lemma~\ref{lem-deriv} shows that $\phit\in C_0^\infty(T(G\setm F))$ 
if and only if
$\phi = \phit\circ T \in C_0^\infty(G\setm F)$.
Lemma~\ref{lem-int-id-ball}, applied to  
$A=U=\{x\in G\setm F: \phi(x)\ne0\}$, then implies
that the integral identity in 
\eqref{eq-weak-sol-p-Lapl} becomes \eqref{eq-div-B-xi-T(G)}.
\end{proof}

\section{Removing the Neumann data}
\label{sect-rem-Neu-data}

Proposition~\ref{prop-Wp-Hp-loc-divB} shows that
the mapping $T$ transforms the unweighted \p-Laplace operator 
from $G$ into the weighted elliptic operator $\Div \B(\xi,\grad\ut)$
on the open upper unit half-ball $T(G)$.

In order to be able to use the theory of Dirichlet problems, developed
for weighted elliptic equations in Heinonen--Kilpel\"ainen--Martio \cite{HKM},
the part of the boundary, where the Neumann data are prescribed, will be
eliminated by reflection in the hyperplane $\{\xi\in\R^n:\xi_n=0\}$.

More precisely, consider the reflection mapping
\[
P\xi = P(\xi',\xi_n) =(\xi',-\xi_n),
\]
and let the open set $D$ consist of $T(G\setm F)$, 
together with its reflection
$PT(G\setm F)$ and the ``Neumann'' part of the boundary
$T(\bdy G\setm F)$ added, that is,
\[
D=B_1\setm \Ft, \quad \text{where } \Ft = T(F)\cup PT(F)\cup\{0\}.
\] 
Clearly, $\Ft$ is closed and hence $D$ is open.

We recall that $T$ maps the base $B'\times\{0\}$ of $G$ onto the upper
unit half-sphere $\{\xi\in\bdy B_1:\xi_n>0\}$
and that the point at infinity in $G$ corresponds to the origin $\xi=0$.
In particular, since we assume that $B'\times\{0\} \subset F$
and $F$ is closed, we have $\bdy D\subset \Ft$. 
Hence, the whole boundary $\bdy D$ will carry a Dirichlet condition.

Now, let $\Tt=P\circ T$ represent the map from the open circular half-cylinder $G$
to the lower unit half-ball $\{\xi\in B_1:\xi_n<0\}$. 
We extend $\B(\xi,q)$ from $T(G)$ to the whole unit ball $B_1$
as follows:
Let $\B(\xi,q)=0$ if $\xi_n=0$, while for $\xi=\Tt(x)$ with $\xi_n<0$ we define
\[ 
\B(\xi,q)=|d\Tt^*(x)q|^{p-2}|J_{\Tt}(x)|^{-1}d\Tt(x)d\Tt^*(x)q. 
\] 
Since $d\Tt(x)=PdT(x)$, $d\Tt^*(x)=dT^*(x)P$ and thus $|J_{\Tt}(x)| = |J_T(x)|$, 
we have  
\begin{equation}   \label{eq-B-Pxi}
\B(\xi,q) = P\B(P\xi,Pq)\quad \text{for all }\xi\in B_1.
\end{equation}
Clearly, Lemma~\ref{lem-JB} holds with $T$  replaced by $\Tt$ as well.
It then immediately follows from Theorems~\ref{thm-B-ellipt} 
and~\ref{thm-B-monot}
that $\B$ satisfies the ellipticity and monotonicity assumptions
(3.3)--(3.7) from Heinonen--Kil\-pe\-l\"ai\-nen--Martio~\cite{HKM} 
in the whole unit ball $B_1$.

The above reflection makes it  possible to remove the Neumann 
boundary data on $T(\bdy G\setm F)$ and obtain an equivalence with 
a Dirichlet problem on $D$. 
First, we make a suitable identification of the function spaces.
Note that Lemmas~\ref{lem-int-id-ball} and~\ref{lem-int-G-T} clearly
hold also with $T$ replaced by $\Tt$.

\begin{defi}   \label{def-L-space-ka}
The space $\Llp_{\ka}(G\setm F)$ consists of all measurable functions
$v$ on $G\setm F$ such that the norm 
\[ 
\|v\|_{\Llp_\ka(G\setm F)} 
= \biggl( \int_{G\setm F} \bigl( |v(x)|^p e^{-p\ka x_n}
     +|\grad v(x)|^p \bigr) \,dx \biggr)^{1/p} < \infty,
\] 
where $\nabla v=(\partial_1v,\cdots,\partial_nv)$ is the 
distributional gradient of $v$.
The space $\Llp_{\ka,0}(\clG\setm F)$ is the completion of 
$C_0^{\infty}(\clG\setm F)$ in the above $\Llp_\ka(G\setm F)$-norm. 
\end{defi}

We alert the reader that the $\Llp_{\ka}(G\setm F)$-norm also includes 
the function $v$, not only its gradient, and because of the weight
$e^{-p\ka x_n}$ it differs from the standard Sobolev norm.
Also note that functions in $\Llp_{\ka,0}(\clG\setm F)$ are required 
to vanish on $F$ (in the Sobolev sense), but not on the 
rest of the lateral boundary $\bdy G\setm F$.
For $t\ge0$, we let
\begin{equation}	\label{eq-Gt}
G_t:=\{x\in\clG:x_n>t\}=\clBprime\times(t,\infty).
\end{equation}
Note that the truncated cylinder $G_t$ is open at its base $B'\times\{t\}$,
but contains the lateral boundary $\bdy B'\times(t,\infty)$.

\begin{lem}           \label{lem-L-ka=L_0} 
Let $v\in\Llp_{\ka}(G\setm F)$. 
Then there exist bounded  $v_j\in\Llp_{\ka}(G\setm F)$ 
with bounded support such that $v_j\to v$ both pointwise a.e.\ and
in $\Llp_{\ka}(G\setm F)$.
\end{lem}

\begin{proof}
Since $v$ can be approximated in the $\Llp_\ka(G\setm F)$ norm by its truncations 
$v_k:=\min\{k,\max\{v,-k\}\}$ at levels $\pm k$, 
we can without loss of generality assume that $v$ is bounded and $|v|\le1$.

For $j=1,2,\cdots$\,, let $v_j= v\eta_j$, where
$\eta_j\in C^{\infty}(\clG)$ is a cut-off function such that 
$0\leq\eta_j\leq 1$ on $\clG$, 
$\eta_j=1$ on $\clG\setm G_j$, $\eta_j=0$ on $G_{2j}$ 
and $|\grad \eta_j|\le 2/j$. 
Then $v_j\in \Llp_\ka(G\setm F)$ with bounded support.
We also have 
\begin{align*}
\|v-v_j\|^p_{\Llp_\ka(G\setm F)}
&= \int_{G_j\setm F} \bigl( |v(1-\eta_j)|^p e^{-p\ka x_n}
             +|\grad(v(1-\eta_j))|^p \bigr) \,dx\\
&\leq  \int_{G_j\setm F}|v|^p e^{-p\ka x_n}\,dx \\
&\quad  +2^p \int_{G_j\setm F} \bigl( |(1-\eta_j)\grad v|^p
    +|v\grad(1-\eta_j)|^p \bigr) \,dx.
 \end{align*}
 Since 
\[
\int_{G_j\setm F} |v\grad(1-\eta_j)|^p 
  \le  \int_{G_j\setm G_{2j}} \biggl( \frac2j \biggr)^p \,dx
  \simle j^{1-p}
\]
and $|(1-\eta_j)\grad v| \le |\grad v|$, we get
\[
\|v-v_j\|^p_{\Llp_\ka(G\setm F)}
\simle \int_{G_j\setm F} \bigl( |v|^p e^{-p\ka x_n} +|\grad v|^p \bigr) \,dx + j^{1-p},
\]
which tends to zero by the dominated convergence theorem
and the assumption $p>1$.
\end{proof}

The following result relates the space $\Llp_\ka(G\setm F)$ to
the weighted Sobolev space on $D$.
We shall write 
\[
\Dplus=\{\xi\in D:\xi_n>0\} = T(G\setm F) \quad \text{and} \quad 
\Dminus=\{\xi\in D:\xi_n<0\} = \Tt(G\setm F).
\]

\begin{prop}   \label{prop-Llp-H}
Assume that $u\in L^1\loc(G\setm F)$ with the distributional
gradient $\grad u \in L^1\loc(G\setm F)$.
Then 
\begin{equation}   \label{eq-comp-Llp-Hp}
\|u\|_{\Llp_\ka(G\setm F)} \simeq \|u\circ T^{-1}\|_{\Hp(\Dplus,\wt)}.
\end{equation}
Moreover, the function 
\begin{equation}  \label{eq-def-ut}
\ut(\xi)= \begin{cases}
        (u\circ T^{-1})(\xi)    &\text{for } \xi\in\Dplus,\\ 
        (u\circ \Tt^{-1})(\xi) &\text{for } \xi\in\Dminus, \end{cases}
\end{equation}
extended arbitrarily to $\xi_n=0$,
belongs to $\Hp(D,\wt)$ if and only if $u\in\Llp_\ka(G\setm F)$, with comparable norms.
\end{prop}

\begin{proof}
The comparison \eqref{eq-comp-Llp-Hp} follows from Lemma~\ref{lem-int-G-T}.
It also shows that $\ut\in\Hp(D,\wt)$ implies that $u\in\Llp_\ka(G\setm F)$.

Conversely, assume that $u\in\Llp_\ka(G\setm F)$. 
Lemma~\ref{lem-int-G-T} (applied to both $T$ and~$\Tt$) implies that 
\[
\ut\in\Hp(\Dplus,\wt) \quad \text{and} \quad
\ut\in\Hp(\Dminus,\wt),
\]
with norms comparable to $\|u\|_{\Llp_\ka (G\setm F)}$.

To see that $\ut\in\Hp(D,\wt)$, let $B\Subset D$ be a ball.
Note that $0\notin D$ and hence $\wt\simeq 1$ in $B$,
with comparison constants depending on $B$.
Because $B\cap\Dplus$ is convex, 
Lipschitz functions are dense in $\Wp(B\cap\Dplus)=\Hp(B\cap\Dplus,\wt)$,
by Maz\cprime ya~\cite[Section~1.1.6]{MazyaSob} or Ziemer~\cite[p.~55]{Zie}.
Reflections of such functions are clearly Lipschitz in $B$.
It then follows that $\ut$ can be approximated in the $\Hp(B,\wt)$-norm by 
Lipschitz functions, and hence $\ut\in\Hp(B,\wt)$.
Since $B$ was arbitrary, we conclude that  $\ut\in\Hp\loc(D,\wt)$.

From~\eqref{eq-comp-Llp-Hp} and a similar comparison for $\Tt(G\setm F)$
we conclude that $\|\ut\|_{\Hp(D,\wt)}$ is finite and \cite[Lemma~1.15]{HKM} 
then shows that $\ut\in\Hp(D,\wt)$.
\end{proof}

\begin{remark}
In Lemma~\ref{lem-cap-B-2B},
we shall see that the origin $0$ has zero 
$(p,\wt)$-capacity in $\Hp(B(0,1),\wt)$ and hence for bounded $F$, we also have
\[
\Hp(D,\wt)=\Hp(D\cup\{0\},\wt)
\quad\text{and}\quad \Hp_0(D,\wt)=\Hp_0(D\cup\{0\},\wt).
\]
In particular, this applies when $F=\clBprime\times\{0\}$ 
is the base of $\clG$, and thus $\Ft=\bdy B_1\cup\{0\}$.
\end{remark}

We also need to compare the spaces of test functions.
Clearly, if $\phit \in C^\infty_0(D)$ then 
$\phit\circ T \in C^\infty_0(\clG\setm F)$, by Lemma~\ref{lem-deriv}.
For Sobolev functions with zero boundary values, we have the following 
statement.

\begin{prop}        \label{prop-L0-to-H0}
If $\vt \in H^{1,p}_0(D,\wt)$, then 
$\vt \circ T \in \Llp_{\ka,0}(\clG\setm F)$.
Conversely, let $u\in\Llp_{\ka,0}(\clG\setm F)$ and define $\ut$ 
as in \eqref{eq-def-ut}.
Then $\ut\in H^{1,p}_0(D,\wt)$. 
\end{prop}

\begin{proof}
To prove the first statement, choose a sequence \(\phit_j\in C_0^\infty(D)\) 
such that \(\phit_j\to \vt\) in \(\Hp_0(D, \wt)\). 
Define $\phi_j=\phit_j\circ T$, with $\phit_j$ restricted to $T(\clG\setm F)$,
and note that $\phi_j\in C_0^\infty(\clG\setm F)$,  by Lemma~\ref{lem-deriv}.
Using \eqref{eq-comp-Llp-Hp} with $u$ replaced by 
$\phi_j-\vt\circ T$, we have 
\begin{align*}
\|\phi_j-\vt\circ T\|_{\Llp_{\ka}(G\setm F)} 
  &\simeq \|\phit_j-\vt\|_{\Hp(\Dplus,\wt)} \\ 
  &\le \|\phit_j-\vt\|_{\Hp(D,\wt)} \to0, 
         \quad \text{as } j\to\infty,  \nonumber
\end{align*}
and consequently, $\vt\circ T \in\Llp_{\ka,0}(\clG\setm F)$.

Conversely, since
$\Llp_{\ka,0}(\clG\setm F)$ is the completion of $C^{\infty}_0(\clG\setm F)$ 
in the $\Llp_{\ka}(G\setm F)$ norm and in view 
of~Proposition~\ref{prop-Llp-H},
we can assume by a density argument that $u\in C^{\infty}_0(\clG\setm F)$. 
Then $\ut$ has compact support in $D$.
Using Lemma~\ref{lem-JB}, it is easily verified that $\ut$, extended continuously when
$\xi_n=0$, is Lipschitz in $D$.
Thus $\ut\in \Hp_0(D,\wt)$ by Lemma~1.25\,(i) in 
Heinonen--Kilpel\"ainen--Martio~\cite{HKM}. 
\end{proof}

For solutions in $\Llp_\ka(G\setm F)$, we
are now able to remove the zero Neumann condition and transfer the
mixed boundary value problem for \eqref{eq-p-lap} in $G\setm F$ to a Dirichlet problem
in $D$.

\begin{thm}   \label{thm-pLapl-after-B}
Assume that $u\in \Llp_\ka(G\setm F)$ is a weak solution of 
$\Delta_p u=0$ in $G\setm F$ with zero Neumann boundary data on 
$\bdy G\setm F$, i.e.~\eqref{eq-weak-sol-p-Lapl} holds. 
Let $\ut$ be as in \eqref{eq-def-ut}.
Then $\ut\in \Hp(D,\wt)$ and for all $\phit\in C_0^\infty(D)$,
\begin{equation}  \label{eq-int-TG=TtG=0}
\int_{\Dplus} \B(\xi,\grad\ut)\cdot\grad \phit \,d\xi 
= \int_{\Dminus} \B(\xi,\grad\ut)\cdot\grad \phit \,d\xi = 0.
\end{equation}
In particular, $\ut$ is a weak solution of the equation
$\Div\B(\xi,\grad \ut(\xi))=0$ in $D$.
\end{thm}

\begin{proof}
Proposition~\ref{prop-Llp-H} implies that $\ut\in \Hp(D,\wt)$.
Let $\phit\in C_0^\infty(D)$. 
The integral identities  in \eqref{eq-int-TG=TtG=0} then follow directly from
\eqref{eq-weak-sol-p-Lapl} and
Lemma~\ref{lem-int-id-ball} with $\phi=\phit\circ T$ 
and $\phi=\phit\circ\Tt$, respectively.
\end{proof}

\begin{remark}
By Theorem~3.70 in \cite{HKM}, weak solutions 
of $\dvg\B(\xi,\grad \ut(\xi))=0$
are (after a modification on a set of zero measure) locally H\"older
continuous in~$D$.
Hence, Theorem~\ref{thm-pLapl-after-B} also implies that $\lim_{x\to x_0} u(x)$
exists and is finite for every $x_0$ belonging to the Neumann boundary 
$\bdy G\setm F$.
\end{remark}

\section{Existence of solutions}
\label{sec-existence}

In this section, we shall prove the existence of weak solutions
to equation \eqref{eq-p-lap} in $G\setm F$ 
with zero Neumann boundary data on \(\bdy G\setm F\) 
and prescribed continuous Dirichlet boundary data $u=f$ on 
\[
F_0=F\cap\bdy(G\setm F).
\]
This will be done using uniform approximations by Lipschitz boundary data
from the space $\Llp_\ka(G\setm F)$.

Let therefore $f\in \Llp_\ka(G\setm F)$ and let $\ft$ be defined as in 
\eqref{eq-def-ut}. 
By Theorems~3.17 and 3.70 in Heinonen--Kilpel\"ainen--Martio~\cite{HKM}, 
there is a unique continuous weak solution $\ut\in\Hp(D,\wt)$ of 
\begin{equation}
\Div(\B(\xi,\grad \ut))=0     \label{eqnB}
\end{equation} 
with boundary data $\ft$ in the sense that $\ut-\ft\in\Hp_0(D,\wt)$.

We shall show that $u=\ut\circ T$ satisfies \eqref{eq-weak-sol-p-Lapl}  
and that $u-f\in\Llp_{\ka,0}(\clG\setm F)$.
To do this, we use the integral formulation
\begin{equation}  \label{Id-B}
\int_{D}\B(\xi,\grad \ut)\cdot\grad\phit\,d\xi=0
\quad \text{for all test functions } \phit\in C^\infty_0(D)    
\end{equation} 
and split the left-hand side into integrals over $\Dplus$ 
and $\Dminus$.
We shall see that the corresponding integrals are the same 
and that each is zero. 
For this, we prove that $\ub=\ut\circ P$ is also a solution 
of \eqref{eqnB} with $\ub-\ft\in\Hp_0(D,\wt)$, and thus $\ut=\ut\circ P$,
by uniqueness. 
The following identity obtained from \eqref{eq-B-Pxi} will be useful, namely
\begin{equation}
\B(\xi,Pq)=P\B(P\xi,q)   	\label{operB-rel}
\end{equation}
for all $q\in\R^n$ and  all $\xi\in B_1$.
First, we prove the following symmetry result.

\begin{lem}				\label{lem-B+=B-}
Let $\ut\in\Hp(D, \wt)$ and define $\ub=\ut\circ P$. 
Let $\phit\in\Hp_0(D,\wt)$ be an arbitrary test function and set 
\(\phib=\phit\circ P\). 
Assume that \eqref{operB-rel} holds in $D$. 
Then for any set $A\subset D$, we have that
\[ 
\int_{A}\B(\xi,\grad \ub(\xi))\cdot \grad\phib(\xi)\,d\xi
 =\int_{P(A)}\B(\xi,\grad \ut(\xi))\cdot 
                     \grad\phit(\xi)\,d\xi.
\]
\end{lem}

\begin{proof}
We use the fact that $\grad\ub(\xi)=P\grad\ut(P\xi)$ and 
\(\grad\phib(\xi)=P\grad\phit(P\xi)\) to rewrite the integral 
on the left-hand side. 
The change of variables $\zeta=P\xi$, together with 
\eqref{operB-rel}, then implies that
\begin{align*}
\int_{A}\B(\xi,\grad \ub(\xi))\cdot\grad\phib\,d\xi 
&=\int_{A}\B(\xi,P\grad \ut(P\xi))\cdot P\grad\phit(P\xi)\,d\xi\\
&=\int_{P(A)}P\B(\zeta,\grad \ut(\zeta))\cdot P\grad\phit(\zeta)\,d\zeta \\
&= \int_{P(A)}\B(\xi,\grad \ut(\xi))\cdot 
                   \grad\phit(\xi)\,d\xi, 
\end{align*}
where in the last step we used the fact that 
$(Pq)\cdot (P\bar{q})=q\cdot \bar{q}$ for any $q,\bar{q}\in\R^n$.
\end{proof}

\begin{cor}        \label{cor-uop}
Assume that $\ft\in\Hp(D,\wt)$   satisfies \(\ft=\ft\circ P\) 
and that \eqref{operB-rel} holds in $D$. 
Let $\ut\in\Hp(D, \wt)$ be a solution of~\eqref{Id-B} with 
\(\ut-\ft\in\Hp_0(D,\wt)\). 
Then \(\ut=\ut\circ P\).
\end{cor}

\begin{proof}
We shall show that $\ub:=\ut\circ P \in\Hp(D, \wt)$ also satisfies
\eqref{Id-B} with the same boundary data.
Let $\phib\in C^\infty_0(D)$ be arbitrary.
Then clearly $\phit:=\phib\circ P \in C^\infty_0(D)$.  
From \eqref{Id-B} and Lemma~\ref{lem-B+=B-} with $A=D=P(D)$ 
we conclude that 
\[ 
\int_{D}\B(\xi,\grad \ub)\cdot\grad\phib\,d\xi
= \int_{D}\B(\xi,\grad \ut)\cdot\grad\phit\,d\xi =0.
\]
Thus $\ub$ is also a solution of \eqref{Id-B} with
\[
\ub-\ft = \ub-\ft\circ P\in\Hp_0(D,\wt). 
\]
By uniqueness of solutions, we get that $\ut=\ub=\ut\circ P$.
\end{proof}

We can now show that \(u:=\ut\circ T\in \Llp_\ka(G\setm F)\) is 
a continuous solution of 
the \p-Laplace equation $\Delta_p u=0$ in $G\setm F$ with
zero Neumann condition on $\bdy G\setm F$ and 
the prescribed Dirichlet boundary data $f$.
Note that since $\ut\in \Hp(D,\wt)$, 
the integral identity \eqref{Id-B} holds for all 
$\phi\in \Hp_0(D,\wt)$, by the density of $C_0^\infty(D)$ 
in $\Hp_0(D,\wt)$.

\begin{thm}   \label{thm-ex-sol-Lip}
For every $f\in \Llp_\ka(G\setm F)$,
there exists a unique continuous weak solution $u\in \Llp_\ka(G\setm F)$ 
of the mixed boundary value problem \eqref{eq-weak-sol-p-Lapl} in $G\setm F$, 
such that \(u-f\in\Llp_{\ka,0}(\clG\setm F)\).

Moreover, the following comparison principle holds: 
If $f_1,f_2\in \Llp_\ka(G\setm F)$ and $f_1\le f_2$ on $F_0$ in the
sense that
$\min\{f_2-f_1,0\}\in \Llp_{\ka,0}(\clG\setm F)$, then the corresponding 
continuous weak
solutions $u_1, u_2\in \Llp_\ka(G\setm F)$ satisfy $u_1\le u_2$ in $G\setm F$.
\end{thm}

\begin{proof}
Let \(\ft\) be the function associated with $f$ as in \eqref{eq-def-ut},
with $u$ replaced by $f$.
Then $\ft\in \Hp(D,\wt)$, by  Proposition~\ref{prop-Llp-H}.
Let $\ut\in\Hp(D,\wt)$ be the unique continuous solution of \eqref{eqnB}
with \(\ut-\ft\in\Hp_0(D,\wt)\), provided by 
Heinonen--Kilpel\"ainen--Martio~\cite[Theorems~3.17 and~3.70]{HKM}.
Define  $u=\ut\circ T$ on $G\setm F$, with $\ut$ restricted to $T(G\setm F)$.

Suppose that \(\phi \in C_0^{\infty}(\clG\setm F) \) is an arbitrary 
test function and set
\begin{equation}  \label{eq-def-phit}
\phit(\xi) = \begin{cases}  
              \phi\circ T^{-1}(\xi) &\text{for }\xi\in D \text{ with }\xi_n\ge0, \\
              \phi\circ\Tt^{-1}(\xi) &\text{for }\xi\in D \text{ with }\xi_n<0.   
      \end{cases} 
\end{equation}
Then  \(\phit\in\Hp_0(D,\wt)\), by Proposition~\ref{prop-L0-to-H0}, 
and clearly $\phit=\phit\circ P$.
From Corollary~\ref{cor-uop} we have that $\ut=\ut\circ P$.
Thus, Lemma~\ref{lem-B+=B-} with $A$ replaced by $\Dplus$ gives
\[
\int_{\Dplus}\B(\xi,\grad \ut)\cdot \grad\phit\,d\xi
=\int_{\Dminus}\B(\xi,\grad \ut)\cdot \grad\phit\,d\xi.
\]
Since the left-hand side of (\ref{Id-B}) is the sum of  
these two integrals, it follows that
\[
\int_{\Dplus}\B(\xi,\grad \ut)\cdot\grad\phit\,d\xi=0.
\]
Lemma~\ref{lem-int-id-ball} now gives
\begin{equation}           \label{pweak-G} 
\int_{G\setm F}|\grad u|^{p-2}\grad u\cdot\grad\phi\,dx
= \int_{\Dplus}\B(\xi,\grad \ut)\cdot\grad\phit\,d\xi =0.		
\end{equation} 
Since $\phi\in C_0^{\infty}(\clG\setm F)$ was assumed to be arbitrary, 
we conclude that $u$ is 
a weak solution of \eqref{eq-weak-sol-p-Lapl} as in Definition~\ref{defi-weak-p-lap}. 
Moreover, $u\in\Llp_\ka(G\setm F)$, by Proposition~\ref{prop-Llp-H}.

To prove uniqueness, suppose that a continuous function
$v\in\Llp_{\ka}(G\setm F)$ satisfies (\ref{eq-weak-sol-p-Lapl})
and \(v-f\in\Llp_{\ka,0}(\clG\setm F)\). 
Let $\vt$ be as in \eqref{eq-def-ut}, with $u$ replaced by~$v$.
Theorem~\ref{thm-pLapl-after-B} then implies that $\vt$ satisfies 
\eqref{Id-B}.
Moreover, Proposition~\ref{prop-L0-to-H0} shows that
$\vt-\ft\in\Hp_0(D, \wt)$. 
From the uniqueness of solutions to \eqref{Id-B} we thus get that
$\vt=\ut$, and so $v=u$.
Finally, the comparison principle follows immediately from
Heinonen--Kilpel\"ainen--Martio~\cite[Lemma~3.18]{HKM}.
\end{proof}

We shall now use uniform approximations to treat continuous boundary
data on~$F_0$.
Suppose that $f\in C(F_0)$ and, if $F_0$ is unbounded, also that 
the limit
\begin{equation}		\label{a-limf}
f(\infty):= \lim_{\substack{x_n\to\infty\\ x\in F_0}} f(x)
\end{equation}
exists and is finite.
Replacing $f$ by $f-f(\infty)$, we can assume without loss of generality that 
$f(\infty)=0$.
We then find a sequence of compactly supported Lipschitz functions $\fb_k:F_0\to\R$ 
such that for $k=1,2,\cdots$\,,
\[
\|\fb_k-f\|_{L^\infty(F_0)}<2^{-k}.   
\]
By the McShane--Whitney extension theorem 
(see Heinonen~\cite[Theorem~2.3]{HJ}), there exist Lipschitz functions 
$f_k:\clG\to \R$ such that $f_k|_{F_0}=\fb_k$.
The Lipschitz constant of $f_k$ is preserved when $\fb_k$ is extended to $\clG$. 
Multiplying $f_k$ by a cut-off function, if necessary, we may assume that 
$f_k$ has compact support.

\begin{thm}			\label{thm-sol-Lip}
Let $f\in C(F_0)$ and, if $F_0$ is unbounded, assume also that the limit 
in \eqref{a-limf} is zero.
Let $\{f_k\}_{k=1}^\infty$ be a sequence of compactly supported 
Lipschitz functions on $\clG$ such that for all $k=1,2,\cdots$\,, 
\begin{equation}   \label{eq-approx-f}
\|f_k-f\|_{L^\infty(F_0)}<2^{-k}.
\end{equation}
Let $u_k\in\Llp_k(G\setm F)$ be the unique continuous weak solution 
of~\eqref{eq-weak-sol-p-Lapl} with $u_k-f_k\in\Llp_{\ka,0}(\clG\setm F)$, 
provided by Theorem~\ref{thm-ex-sol-Lip}.
Then $u_k$ converge uniformly in $G\setm F$
and the function $u:=\lim_{k\to\infty}u_k$ is a bounded
continuous weak solution 
of the \p-Laplace equation~\eqref{eq-p-lap}
in $G\setm F$ with zero Neumann boundary data 
on~$\bdy G\setm F$, in the sense of~\eqref{eq-weak-sol-p-Lapl}.
\end{thm}

\begin{proof}
For $k=1,2,\ldots$\,, note that $f_k\in \Llp_k(G\setm F)$ and define
\begin{align} \label{eq-fk-Lip}
\ft_k(\xi) &= \begin{cases}
(f_k\circ T^{-1})(\xi)   
     &\text{for } \xi\in \clB_1\setm\{0\}\text{ with }\xi_n\ge0,\\ 
(f_k\circ \Tt^{-1})(\xi)  &\text{for } \xi\in \clB_1\text{ with }\xi_n<0,  \\
     0 &\text{for }\xi=0, 
         \end{cases}
\end{align} 
and similarly for $\xi\in\bdy D$ define $\ft$ in terms of $f$ as in \eqref{eq-fk-Lip}.
Then the sequence $\{\ft_k\}_{k=1}^\infty$ converges uniformly to $\ft$ on $\bdy D$, 
i.e.\ for all $k=1,2,\cdots$\,,
\begin{equation*} 
\ft_k-2^{-k}\leq \tilde{f}\leq \ft_k+2^{-k}	\quad \text{on } \bdy D.	
\end{equation*} 
Recall that in the proof of Theorem~\ref{thm-ex-sol-Lip}, we have $u_k=\ut_k\circ T$ 
with $\ut_k$ restricted to $T(G\setm F)$, where $\ut_k\in\Hp(D,\wt)$ 
is the solution of~\eqref{eqnB} in $D$ such that $\ut_k-\ft_k\in\Hp_0(D,\wt)$.
Then $\ut_k+3\cdot2^{-k}$ and $\ut_k-3\cdot2^{-k}$ are solutions of \eqref{eqnB} 
in $D$ with boundary data $\gt_k=\ft_k+3\cdot2^{-k}$ and 
$\gb_k=\ft_k-3\cdot2^{-k}$, respectively.
Moreover, $\gt_k$ is decreasing to $\ft$ and $\gb_k$ is increasing to $\ft$ on $\bdy D$.

By the comparison principle \cite[Lemma~3.18]{HKM}, the sequence 
$\ut_k+3\cdot2^{-k}$ is decreasing to a function $\ut$ in $D$, 
while the sequence
$\ut_k-3\cdot2^{-k}$ is increasing to $\ut$. 
Clearly, the convergence is uniform.
Since $\ft_k$ are bounded, so are $\ut_k$ by the maximum principle.
Hence also the functions $u_k$ converge uniformly to 
the bounded continuous function $u=\ut\circ T$ in $G\setm F$.
The Harnack convergence theorem \cite[Theorem~6.13]{HKM} implies that
$\ut$ is a solution of $\Div(\B(\xi,\grad \ut))=0$ in $D$.
In particular, $\ut\in \Hp\loc(D,\wt)$ and \eqref{Id-B} holds for all
$\phit\in C^\infty_0(D)$ and, by a density argument, also for all $\phit\in\Hp(D,\wt)$
which have compact support in $D$. 
Since $\ut\in \Hp\loc(D,\wt)$, 
it follows from Lemma~\ref{lem-int-G-T} that
$u\in \Wp(U)$ for every open set $U\Subset G\setm F$
and hence $u\in\Wploc(G\setm F)$.

Finally, we show that $u$ satisfies \eqref{eq-weak-sol-p-Lapl}.
Let $\phi\in C^\infty_0(\clG\setminus F)$ and define
$\phit$ as in \eqref{eq-def-phit}. 
By Proposition~\ref{prop-L0-to-H0}, the function $\phit$ belongs to $\Hp_0(D,\wt)$ 
and has compact support in $D$.
As in the proof of Theorem~\ref{thm-ex-sol-Lip}, we can therefore conclude 
from Lemmas~\ref{lem-int-id-ball} 
and~\ref{lem-B+=B-} that \eqref{pweak-G}, and thus \eqref{eq-weak-sol-p-Lapl}, holds for all 
$\phi\in C^\infty_0(\clG\setminus F)$, i.e.\ that $u$ is a weak solution
of the \p-Laplace equation \eqref{eq-p-lap} in $G\setm F$
with zero Neumann boundary data on $\bdy G\setm F$.
\end{proof}

We shall now see that the function $u$ obtained in Theorem~\ref{thm-sol-Lip} attains its 
continuous boundary data on $F_0$, except for a set of zero \p-capacity.
The definition below follows Chapter~2 in Heinonen--Kilpel\"ainen--Martio~\cite{HKM}.

\begin{defi}       \label{def-cap-p-wt}
Suppose that $K$ is a compact subset of an open set $\Om\subset\R^n$. 
The \emph{variational $(p,\wt)$-capacity} of $K$ in $\Om$, is 
\begin{equation}   \label{def-wcap}
\cp_{p,\wt}(K,\Om)=\inf_v\int_{\Om}|\grad v|^p\wt(\xi)\,d\xi,			
\end{equation}
 where the infimum is taken over all $v\in C_0^{\infty}(\Om)$ satisfying 
$v\geq 1$ on $K$.
\end{defi}

By a density argument, the infimum in \eqref{def-wcap} can 
equivalently be taken over all
$v\in\Hp_0(\Om,\wt)\cap C(\Om)$ such that $v\geq1$ on $K$, 
see \cite[pp.~27--28]{HKM}.
The capacity $\cp_{p,\wt}$ is extended using a standard procedure to open 
and then to arbitrary sets, see \cite[p.~27]{HKM}.
By Theorem~2.5 in \cite{HKM}, it 
is a Choquet capacity and for all Borel  
(even Suslin) sets $E\subset\Om$, 
\begin{equation}   \label{eq-extend-cap-wt}
\cp_{p,\wt}(E, \Om)
=\sup   \{\cp_{p,\wt}(K,\Om): K\subset E \text{ compact}\}.
\end{equation}
We say that a set $E\subset\R^n$ is of $(p,\wt)$-capacity zero if 
$\cp_{p,\wt}(E\cap\Om, \Om)=0$ 
for every bounded open set $\Om\subset\R^n$.

In Heinonen--Kilpel\"ainen--Martio~\cite[p.~122]{HKM}, 
a point $\xi_0\in\bdy D$ is called \emph{regular}
for the equation $\Div(\B(\xi,\grad \ut))=0$
if for every boundary data 
$\fh\in\Hp(D,\wt)\cap C(\clD)$, the solution $\uh$ of~\eqref{eqnB}
with $\uh-\fh \in\Hp_0(D,\wt)$ satisfies
\begin{equation}   \label{eq-def-reg-HKM}
\lim_{D\ni\xi\to\xi_0} \uh(\xi) = \fh(\xi_0).
\end{equation}
The fact that the set of irregular boundary points has zero capacity (by 
the Kellogg property \cite[Theorem~8.10]{HKM})
now makes it possible to obtain the  
precise existence result for the Zaremba problem~\eqref{eq-weak-sol-p-Lapl}
with continuous Dirichlet boundary data,
formulated in Theorem~\ref{thm-cont}.

Recall that the Sobolev $C_p$-capacity is the capacity associated with the
usual Sobolev space $\Wp(\R^n)$ and is for compact sets defined as
\begin{equation}		\label{eq-Cp}
C_p(K)=\inf_v\int_{\R^n}(|v|^p+|\grad v|^p)\,dx,
\end{equation}
where the infimum is taken over all $v\in C_0^\infty(\R^n)$ such that $v\ge1$ 
on $K\subset R^n$, see  \cite[Section~2.35 and Lemma~2.36]{HKM}.
Similarly to $\cp_{p,\wt}$, it extends to general sets as a Choquet capacity and
\begin{equation}		\label{eq-Cp-sup}
C_p(E) = \sup\{C_p(K): K\subset E \text{ compact}\} 
\quad \text{for all Borel } E\subset \R^n.
\end{equation}

\begin{lem}  \label{lem-zero-cap}
Let $Z\subset T(\clG)$ be a set of $(p,\wt)$-capacity zero.
Then $C_p(T^{-1}(Z))=0$.
\end{lem}

We will not need it, but it is not difficult to show that the converse 
of Lemma~\ref{lem-zero-cap} is also true.
For $Z\subset T(G_0)$ it also follows from Lemmas~\ref{lem-compare-cap} and \ref{lem-cap-Cp}. 

\begin{proof}
Because of \eqref{eq-Cp} and \eqref{eq-Cp-sup}, it suffices to show 
that for every compact set 
$K\subset T^{-1}(Z)$ there are $v_j\in C_0^\infty(\R^n)$ such that $v_j\ge1$ on $K$ and
$\|v_j\|_{\Wp(\R^n)} \to 0$ as $j\to\infty$.
We therefore  choose a bounded open set $\Om\supset K$.
Then $T(K)$ is also compact and $T(K)\subset Z$.
Moreover, $T(\Om)\supset T(K)$ is a bounded open set.

Since $Z$ is of $(p,\wt)$-capacity zero, we have $\cp_{p,\wt}(T(K),T(\Om))=0$
and hence we can find functions 
$0\le\vh_j\in C^\infty_0(T(\Om))$ satisfying
\[
\vh_j\ge 1 \quad \text{on } T(K) \quad \text{and} \quad
\int_{T(\Om)} |\grad \vh_j|^p \wt \,d\xi \to 0, \quad j\to\infty.
\]
The Poincar\'e inequality~\cite[(1.5)]{HKM} implies that also
\[
\int_{T(\Om)} |\vh_j|^p \wt \,d\xi \le C
\int_{T(\Om)} |\grad \vh_j|^p \wt \,d\xi \to 0, \quad j\to\infty,
\]
where the constant $C$ depends on $T(\Om)$.
Now, because $T$ is a smooth diffeomorphism by Lemma~\ref{lem-deriv}, 
letting $v_j=\vh_j\circ T$
provides us with  functions $0\le v_j\in C^\infty_0(\Om)$ such that
$v_j\ge 1$ on $K$.

Lemma~\ref{lem-int-G-T}, together with the fact that $e^{-p\ka x_n}\simeq 1$ on the 
bounded set $\Om$, implies that
\begin{align*}
\|v_j\|_{\Wp(\R^n)} &\simeq
\int_\Om\bigl(|v_j|^p e^{-p\ka x_n} +|\grad v_j|^p\bigr)\,dx \\
&\simeq\int_{T(\Om)}\bigl(|\vh_j|^p +|\grad\vh_j|^p\bigr)\wt(\xi)\,d\xi \to 0, 
\quad j\to\infty,
\end{align*}
with comparison constants depending on $\Om$.
Thus, $K$ (and consequently $T^{-1}(Z)$) has zero Sobolev $C_p$-capacity.
\end{proof}

\begin{proof}[Proof of Theorem~\ref{thm-cont}]
The function 
\begin{equation}   \label{eq-u-f(infty)}
u:= f(\infty) + \lim_{k\to\infty} u_k,
\end{equation}
provided by Theorem~\ref{thm-sol-Lip},
satisfies \eqref{eq-weak-sol-p-Lapl}.
By considering $f-f(\infty)$ and $u-f(\infty)$ instead of $f$
  and $u$, respectively, we
can assume without loss of generality that $f(\infty)=0$.

Let $Z\subset \bdy D$ be the set of irregular boundary points for the
equation~\eqref{eqnB}.
The Kellogg property \cite[Theorem~8.10]{HKM} and 
Lemma~\ref{lem-zero-cap} imply that the set $Z_0:=T^{-1}(Z\cap T(\clG))$ 
has zero Sobolev  $C_p$-capacity.

It remains to show that \eqref{eq-lim=f-qe} holds for all $x_0\in F_0\setm Z_0$.
Let therefore $\eps>0$ be arbitrary.
Then $\xi_0:= T(x_0)\neq 0$ is a regular boundary point of $D$ for the 
equation~\eqref{eqnB}.  
Recall from Theorem~\ref{thm-sol-Lip} and its proof that $u=\ut\circ T$, where
$\ut$ is the uniform limit of solutions $\ut_k$ to~\eqref{eqnB} 
in $D$ with Lipschitz boundary data $\ft_k$
such that $\ft_k\to \ft$ uniformly on  $\bdy D$, where $\ft$
is defined in terms of $f$ as in \eqref{eq-fk-Lip}.
Thus, we can find $k$ so that
\[
\|\ut_k-\ut\|_{L^\infty(D)} < \eps
\quad \text{and} \quad
\|\ft_k(\xi_0)- f(x_0)\|_{L^\infty(D)} < \eps.
\]
Since $\xi_0$ is a regular boundary point for \eqref{eqnB}, there is a neighbourhood
$V\subset T(\R^n)$ of $\xi_0$ such that 
$|\ut_k-\ft_k(\xi_0)|<\eps$ in $V\cap D$.
The triangle inequality then implies that for all 
$x\in T^{-1}(V)\cap (G\setm F)$,
\begin{align*}
&|u(x)-f(x_0)| = |\ut(T(x)) - f(x_0)| \\
&\quad \quad \quad \le |\ut(T(x))-\ut_k(T(x))| + |\ut_k(T(x))-\ft_k(\xi_0)| 
        + |\ft_k(\xi_0)- f(x_0)| < 3\eps.
\end{align*}
Since $\eps>0$ was arbitrary, this shows that \eqref{eq-lim=f-qe} holds.

Finally,  the  continuity of $\ut$ in~$D$ shows that the limit
$\lim_{x\to x_0} u(x)$
exists and is finite for every $x_0\in\bdy G\setm F$.
\end{proof}

The proof of Theorem~\ref{thm-sol-Lip}, together with 
\eqref{eq-u-f(infty)}, also leads to the following comparison
principle.

\begin{cor}   \label{cor-com-princ}
If $f, h\in C(F_0)$ and $f \le h$, then the corresponding continuous weak
solutions
$u$ and~$v$, provided by Theorem~\ref{thm-cont}, satisfy $u\le v$
in $G\setm F$.
\end{cor}

\begin{proof}
By \eqref{eq-approx-f}, the functions $f_k$ and $h_k$, uniformly
approximating $f-f(\infty)$ and $h-h(\infty)$ in 
Theorem~\ref{thm-sol-Lip}, satisfy for all $k=1,2,\ldots,$
\[
f_k+f(\infty) \le f+2^{-k} \le h+2^{-k} \le h_k+h(\infty)+2^{1-k}
\quad \text{on } F_0.
\]
The comparison principle in Theorem~\ref{thm-ex-sol-Lip}
then shows that also
the continuous solutions $u_k$
and $v_k$ with $u_k-f_k\in \Llp_{\ka,0}(\clG\setm F)$ and $v_k-h_k\in \Llp_{\ka,0}(\clG\setm F)$ satisfy
\[
u_k+f(\infty) \le v_k+h(\infty)+2^{1-k}  \quad \text{in } G\setm F.
\]
Letting $k\to\infty$, together with \eqref{eq-u-f(infty)}, concludes
the proof.
\end{proof}

 \section{Capacity estimates}
\label{sect-cap}

In this section we compare the variational capacity $\cp_{p,\wt}$
from Definition~\ref{def-cap-p-wt} with a new variational capacity defined
on the cylinder $G$ and adapted to the mixed boundary value problem.
These capacities will play an essential role for the boundary regularity
at infinity.

Recall from \eqref{eq-Gt} that for $t\ge0$, 
$
G_t:=\{x\in \clG: x_n>t\} = \clBprime\times (t,\infty).$
Note that $G_t$ contains the lateral boundary, but not the base
$\clBprime\times\{t\}$, of the truncated cylinder $B'\times (t,\infty)$.
It can also be written as
\[
G_t = \clG\setm \bigl(\clBprime\times[0,t]\bigr).
\]
The results from the previous sections concerning function spaces on
$\clG\setm F$ are therefore available for $G_t$ by 
replacing $F$ with 
\[
Q_t:=\clBprime\times[0,t] = \clG\setm G_t.
\]
Note that 
\[
T(G_t)= \{\xi\in B_r: \xi_n \ge0\} \setm \{0\}
\]
is the upper half of the ball $B_r$, with the origin removed, where $r= e^{-\ka t}$.

Inspired by \eqref{def-wcap},  we define the following variational \p-capacity on $G_t$. 

\begin{defi}  
Let $E\subset G_t$, where $t\ge0$.
The \emph{{\rm(}Neumann\/{\rm)} variational \p-capacity} of $E$ with respect to $G_t$ is 
\begin{equation}
\cp_{p,G_t}(E)= \inf_v \int_{G_t}|\grad v|^p\,dx,	
\label{def-Gscap}
\end{equation} 
where the infimum is taken over all functions 
$v\in \Llp_{\ka,0}(\clG\setm Q_t)$ satisfying $v\ge1$ in $G_t\cap U$ for some 
open neighbourhood $U$ of $E$.
\end{defi}

It follows directly from the definition that $\capt$ is an outer
capacity, i.e.\ for every $E\subset G_t$,
\begin{equation}   \label{eq-outer-capt}
\capt(E) = \inf  \{\capt(G_t\cap U): U\supset E \text{ open}\}.
\end{equation}
It is also clearly a monotone set function, i.e.\ 
$\capt(E_1)\le\capt(E_2)$ whenever $E_1\subset E_2\subset G_t$.
The subadditivity 
\[
\capt(E_1\cup E_2) \le \capt(E_1)+\capt(E_2)
\]
also follows directly by considering the function $\max\{v_1,v_2\}$, with $v_j$
admissible for $\capt(E_j)$, $j=1,2$.
By truncation, the admissible functions $v$ in \eqref{def-Gscap} can be assumed 
to satisfy $0\le v\le1$.

As in \cite[pp.\ 27--28]{HKM}, the following approximation argument allows us to test 
the capacity of compact sets with smooth admissible functions.
Recall the definition of $C^\infty_0(\clG\setm Q_t)$ and $\Llp_{\ka,0}(\clG\setm Q_t)$
in \eqref{eq-def-C0-infty} and Definition~\ref{def-L-space-ka}.
That is, in \eqref{def-Gscap} we have
$v(x)=0$ when  $x_n\le t$ and when $x_n$ is sufficiently large, 
but there is no such requirement on the lateral boundary of $G_t$.

\begin{lem}  \label{lem-cap-K-Cinfty}
If $K\subset G_t$ is compact, then the infimum in \eqref{def-Gscap}
can equivalently be taken over all $v\in C_0^\infty(\clG\setm Q_t)$
such that $v=1$ on $K$.
\end{lem}

\begin{proof}
Denote the latter infimum by $I$.
Let $v\in\Llp_{\ka,0}(\clG\setm Q_t)$ be
such that $0\le v\le1$ in $G_t$ and $v=1$ in $G_t\cap U$ for some 
bounded open set $U\supset K$.
Fix a cut-off function $\eta\in C_0^\infty(U)$ such that $\eta=1$ on $K$.
Let $v_j\in C_0^\infty(\clG\setm Q_t)$ be such that $v_j\to v$ in 
$\Llp_{\ka,0}(\clG\setm Q_t)$. 
Then it is easily verified that the functions 
\[
u_j:=\eta v + (1-\eta)v_j = v+(1-\eta)(v_j-v)
\] 
belong to 
$C_0^\infty(\clG\setm Q_t)$ and satisfy $u_j=1$ on $K$.
We therefore have
\begin{align*}
I\le \|\grad u_j\|_{L^p(G_t)} 
&\le \|\grad v\|_{L^p(G_t)} + \|\grad((1-\eta)(v_j-v))\|_{L^p(G_t)} \\
&\le \|\grad v\|_{L^p(G_t)} + \|\grad(v_j-v)\|_{L^p(G_t)} + C\|v_j-v\|_{L^p(U)} \\
&\le \|\grad v\|_{L^p(G_t)} + C\|v_j-v\|_{\Llp_{\ka}(G_t)},
\end{align*}
where $C$ depends on $U$ and $\eta$.
Letting $j\to\infty$ and then taking infimum over all $v$ admissible
in the definition of $\capt(K)$ shows one inequality.
The opposite inequality is straightforward.
\end{proof}

For monotone sequences of sets, the capacity $\capt$ has the following continuity properties, 
which show that it is a Choquet capacity. 

\begin{lem}
If $K_j\searrow K=\bigcap_{j=1}^\infty K_j$ is a decreasing sequence of compact subsets of $G_t$ then
\[
\capt(K) = \lim_{j\to\infty} \capt(K_j).
\]
\end{lem}

\begin{proof}
This follows immediately from the monotonicity of $\capt$ and from 
\eqref{eq-outer-capt} since for every
open $U\supset K$, there is some $j$ such that $K_j\subset \clG \cap U$ and hence
\[
\capt(K) \le \lim_{j\to\infty} \capt(K_j) \le \capt(\clG \cap U).\qedhere
\]
\end{proof}

\begin{prop}  
If $E_j\nearrow E=\bigcup_{j=1}^\infty E_j$ is an increasing sequence of arbitrary subsets of $G_t$ then
\[
\capt(E) = \lim_{j\to\infty} \capt(E_j).
\]
\end{prop}

\begin{proof}
The proof follows the arguments from Kinnunen--Martio~\cite{KiMaNov}.
The inequality $\lim_{j\to\infty} \capt(E_j)\le \capt(E)$ follows immediately from the monotonicity of $\capt$.
For the opposite inequality, let $u_j\in \Llp_{\ka,0}(\clG\setm Q_t)$ be such that
$0\le u_j\le1$ in $G_t$, 
$u_j=1$ in $\clG \cap U_j$ for some open neighbourhood $U_j$ of $E_j$ and 
\[
\|\grad u_j\|^p_{L^p(G_t)} \le \capt(E_j) + 2^{-j}.
\]
We can assume that $\lim_{j\to\infty} \capt(E_j)<\infty$ and hence the
sequence $u_j$ is bounded in $\Llp_{\ka,0}(\clG\setm Q_t)$.
Since the space $L^p(G_t,e^{-p\ka x_n}\,dx)\times L^p(G_t)$ is reflexive, 
there is a subsequence (also denoted $u_j$) such that 
\[
(u_j,\grad u_j)\to(u,\grad u) \quad \text{weakly in }
L^p(G_t,e^{-p\ka x_n}\,dx)\times L^p(G_t).
\]
Mazur's lemma (see e.g.\ \cite[Lemma~1.29]{HKM}) applied to each of the subsequences 
$\{(u_i,\grad u_i)\}_{i\ge j}$ 
in $L^p(G_t,e^{-p\ka x_n}\,dx)\times L^p(G_t)$ provides us with
finite convex combinations $v_j$ of $\{u_i\}_{i\ge j}$ satisfying
\[
\|v_j-u\|_{\Llp_{\ka}(G_t)}<2^{-j}, \quad j=1,2,\ldots.
\]
Note that $v_j\in\Llp_{\ka,0}(\clG\setm Q_t)$
and $v_j\ge 1$ in $\clG\cap V_j$ for some open neighbourhood $V_j$ of $E_j$, which is obtained as a finite intersection 
of the open neighbourhoods $U_i$ of $E_i$, $i\ge j$.

It follows that $w_j:=v_j+\sum_{i\ge j}|v_{i+1}-v_i|\in \Llp_{\ka,0}(\clG\setm Q_t)$
and $w_j\ge v_k\ge1$ in $\clG\cap V_k$ for each $k\ge j=1,2,\ldots$\,.
Hence $w_j\ge1$ in $\clG\cap V$ for the open neighbourhood
$V=\bigcap_{k=j}^\infty V_k$ of $E$. 
So it is admissible in the definition of $\capt(E)$ and hence
\begin{align*}
\capt(E)^{1/p} &\le \|\grad w_j\|_{L^p(G_t)} 
\le \|\grad v_j\|_{L^p(G_t)} + \sum_{i\ge j} \| \grad(v_{i+1}-v_i)\|_{L^p(G_t)} \\
&\le \|\grad v_j\|_{L^p(G_t)} + \sum_{i\ge j} (2^{-i-1}+2^{-i})
\le \capt(E_j) + 2^{2-j}.
\end{align*}
Letting $j\to\infty$ concludes the proof.
\end{proof}

The Choquet capacitability theorem (see Choquet~\cite[Th\'eor\`eme~1]{choquet}) now implies that 
all Borel sets $E\subset G_t$ are capacitable, i.e.\ 
\begin{equation}   \label{eq-extend-cap-Gt}
\cp_{p,G_t}(E) = \sup   \{\cp_{p,G_t}(K): K\subset E \text{ compact}\}.
\end{equation}

We shall now compare the two variational capacities $\cp_{p,\wt}$ and $\cp_{p,G_t}$.

\begin{lem}   \label{lem-compare-cap}
There exist constants $C',C''>0$, independent of $t\ge0$, 
such that for all Borel sets $E\subset G_t$, 
\[
C' \cp_{p,G_t}(E) \le \cp_{p,\wt}(\Et,B_r) \le C'' \cp_{p,G_t}(E),
\] 
where $\Et= T(E)\cup\Tt(E)$ and $r=e^{-\ka t}$.
\end{lem}

\begin{proof}
To prove the first inequality, 
let $\vb\in \Hp_0(B_r,\wt)\cap C(B_r)$ be such that $\vb\ge1$ on $\Et$.
By considering the open sets $\{\xi\in B_r:\vb(\xi)>1-\eps\}$ and letting $\eps\to0$, 
we can assume that $\vb\ge1$ in an open neighbourhood of $\Et$.
Letting $v=\vb\circ T$ on $G_t$, we have from Proposition~\ref{prop-L0-to-H0} and
Lemma~\ref{lem-int-G-T} that $v\in \Llp_{\ka,0}(\clG\setm Q_t)$ and 
\[ 
\int_{G_t}|\grad v|^p dx  
     \simeq \int_{T(G_t)}|\grad\vb|^p\wt(\xi)\,d\xi 
    \le \int_{B_r}|\grad\vb|^p\wt(\xi)\,d\xi. 
\] 
Since $v$ is admissible for $\capt(E)$, taking infimum over all $\vb$
admissible in the definition of $\cp_{p,\wt}(\Et,B_r)$
proves the first inequality in the lemma.

For the second inequality, we need continuous test functions in \eqref{def-wcap}.
Let therefore $\eps>0$ and using \eqref{eq-extend-cap-wt} choose 
a compact set $\Khat\subset \Et$ such that
\[
\cp_{p,\wt}(\Et,B_r) \le \cp_{p,\wt}(\Khat,B_r) +\eps
\quad \text{if } \cp_{p,\wt}(\Et,B_r)<\infty,
\]
and $\cp_{p,\wt}(\Khat,B_r)>1/\eps$ otherwise.
Replacing $\Khat$ by its symmetrization $\Khat\cup P(\Khat)$ and noting
that $0\notin \Et$, we can assume that $\Khat= T(K)\cup\Tt(K)$ for
some compact set $K\subset E$.
Now, use Lemma~\ref{lem-cap-K-Cinfty} to find $v\in C^\infty_0(\clG\setm Q_t)$ satisfying $v\ge1$ on $K$ and
\begin{equation}   \label{eq-cap-epsilon}
\int_{G_t} |\grad v|^p\,dx  \le \cp_{p,G_t}(K)+\eps.	
\end{equation}   
For $\xi\in B_r,$ define $\vt$ as in \eqref{eq-fk-Lip} with $f_k$ replaced
by $v$.
As $\vt=\vt\circ P$, applying Lemma~\ref{lem-int-G-T} together with \eqref{eq-cap-epsilon}
then gives
\[
\int_{\Tt(G_t)}|\grad\vt|^p\wt(\xi)\,d\xi 
    = \int_{T(G_t)}|\grad\vt|^p\wt(\xi)\,d\xi
    \simeq \int_{G_t}|\grad v|^p\,dx \le \cp_{p,G_t}(K)+\eps.
\]
Since $B_r=T(G_t)\cup\Tt(G_t)\cup\{0\}$ and 
$\vt \in \Hp_0(B_r,\wt)\cap C(B_r)$ is admissible for $\cp_{p,\wt}(\Khat,B_r)$ as 
in \eqref{def-wcap}, we have 
\begin{align*}
\cp_{p,\wt}(\Et,B_r) &\le \cp_{p,\wt}(\Khat,B_r) + \eps
\le \int_{B_r}| \grad\vt|^p\wt(\xi)\,d\xi + \eps\\
&\simle \cp_{p,G_t}(K)+\eps \le \cp_{p,G_t}(E)+\eps,
\end{align*}
if $\cp_{p,\wt}(\Et,B_r)<\infty$, and 
\[
1/\eps < \cp_{p,\wt}(\Khat,B_r) \simle \cp_{p,G_t}(E)+\eps
\]
otherwise.
Letting $\eps\rightarrow 0$ completes the proof.
\end{proof}

The following simple capacity estimates for spherical condensers 
will be useful when dealing with the Wiener criterion.

\begin{lem}   \label{lem-cap-B-2B}
For all\/ $0<r<R<\infty$, 
\begin{align}
\cp_{p,\wt}(\clB_r,B_{2r})  &= \cp_{p,\wt}(B_r,B_{2r}) \simeq 1, 
\nonumber \\
\cp_{p,\wt}(B_r,B_R)  &\simle C\biggl(\log\frac{R}{r}\biggr)^{1-p},
\label{eq-cap-cond}
\end{align}
with comparison constants  depending only on $n$ and $p$.

In particular, the origin $0$ has zero $(p,\wt)$-capacity.
\end{lem}

\begin{proof}
Lemma~\ref{lem-w-Ap} with $\al=1$ gives $\mu(B_r)\simeq r^p$.
Hence, by \cite[Lemma~2.14]{HKM} we have
\[
\cp_{p,\wt}(B_r,B_{2r})\simeq r^{-p}\mu(B_r)\simeq 1. 
\]
The equality for $\clB_r$ follows from \cite[(6.40)]{HKM}.
For \eqref{eq-cap-cond} note that the function
\[
v(\xi) = \begin{cases}
      1  & \text{if } |\xi|\le r, \\
      a\log\frac{R}{|\xi|} & \text{if } r< |\xi|<R, \\
      0 & \text{if } |\xi|\ge R,   \end{cases}
\quad \text{with} \quad a = \biggl(\log\frac{R}{r}\biggr)^{-1}
\]
is admissible for the condenser $(B_r,B_R)$.
The change of variables $\rho=|\xi|$ yields
\[
\cp_{p,\wt}(B_r,B_R) \le \int_{B_R\setm B_r}|\grad v|^p\wt(\xi)\,d\xi
= a^p \int_{B_R\setm B_r} \frac{d\xi}{|\xi|^n}
\simeq a^p\int_r^R\,\frac{d\rho}{\rho}= a^p\log\frac{R}{r}.
\]
The last statement follows by letting $r\to0$.
\end{proof}

We end this section with the following two lemmas which give a comparison 
between the capacities $C_p$ from \eqref{eq-Cp} and $\cp_{p,G_{t-1}}$ from \eqref{def-Gscap}.

\begin{lem}		\label{lem-cap-Cp}
Let $0\le s<t$ and $E\subset \clG_t$. 
Then $\cp_{p,G_{s}}(E)\simle \min\{1,C_p(E)\}$ with the comparison constant depending on
$t-s$.
\end{lem}

\begin{proof}
Using~\eqref{eq-Cp-sup} and \eqref{eq-extend-cap-Gt},
we may without loss of generality assume that $E$ is a compact subset of $\clG_t$.
Let $v\in C_0^\infty(\R^n)$ be such that $v\ge1$ on $E$.

Choose a cut-off function $\eta\in C^\infty(\clG)$,  such that 
$0\leq\eta\leq 1$ on $\clG$, $\eta=1$ on $G_t$, 
$\eta=0$ on $G\setm G_{s}$ and $|\grad \eta|\le 2/(t-s)$. 
Then both $\eta v$ and $\eta$ are admissible for $\cp_{p,G_{s}}(E)$ and so 
\begin{align*}
\cp_{p,G_{s}}(E)&\le\int_{G_{s}}|\grad(\eta v)|^p\,dx
\simle\int_{G_{s}}(|v\grad\eta |^p+|\eta\grad v|^p)\,dx\\
&\simle\int_{\R^n}(|v|^p+|\grad v|^p)\,dx.
\end{align*}
Since also
\[
\cp_{p,G_{s}}(E) \le \int_{G_{s}}|\grad\eta|^p\,dx \simle 1,
\]
the statement follows by taking the infimum over all $v\in C_0^\infty(\R^n)$ 
admissible in the definition of $C_p(E)$.

\end{proof}

\begin{lem} \label{lem-Cp-cap}
Let $E\subset \clG_t\setm G_{t+1}$ with $t\ge1$. 
Then $C_p(E)\simle\cp_{p,G_{t-1}}(E)$.
\end{lem}

\begin{proof}
Without loss of generality, we may assume that $E$ is a compact subset 
of $\clG_t\setm G_{t+1}$.
Let $v\in C_0^\infty(\clG\setm Q_{t-1})$ with $v\ge1$ on $E$.
Extend $v$ to a function $\vb$ on the larger cylinder 
\[
G'=B'(0,2)\times(0,\infty), \quad \text{where }
B'(0,2)=\{x'\in\R^{n-1}:|x'|<2\},
\]
so that $\vb=v$ on $G$ and 
\[
\|\vb\|_{\Wp(G')}\simle \|v\|_{\Wp(G)},
\]
where the comparison constant in $\simle$ depends on $p$ and $n$.
This is possible since $B'\subset\R^{n-1}$ is an extension domain, 
and can be achieved e.g.\ by the spherical inversion
\[
\vb(x',x_n)=v\biggl(\frac{x'}{|x'|^2},x_n\biggr) \quad\text{when }1\le|x'|\le2.
\]
Multiply $\vb$ by a cut-off function $\eta\in C_0^\infty(B'(0,2)\times(t-1,t+2))$ 
such that $\eta=1$ on $G_t\setm G_{t+1}$, $0\le\eta\le1$ and $|\grad\eta|\le2$.
Then $\eta\vb$ is admissible for $C_p(E)$ and so 
\begin{align} \label{eq-Cp-int}
C_p(E)&\le \int_{\R^n}(|\eta\vb|^p+|\grad(\eta\vb)|^p)\,dx \nonumber\\
&\simle \int_{G'}(|\vb|^p+|\grad\vb|^p)\,dx
\simle \int_{G_{t-1}}(|v|^p+|\grad v|^p)\,dx.
\end{align}
Now, define $\vt$ as in \eqref{eq-fk-Lip} with $f_k$ replaced by $v$. 
Then $\vt\in\Hp_0(B_\rho,\wt)$ with $\rho=e^{-\ka(t-1)}$.
Using Lemma~\ref{lem-int-G-T} and the weighted Poincar\'e inequality
\cite[(1.5)]{HKM}, we get that
\begin{align*}
\int_{G_{t-1}}|v|^p\,dx &\simle e^{p\ka t}\int_{B_{\rho}}|\vt|^p\wt(\xi)\,d\xi\\
&\simle e^{p\ka t}\rho^p\int_{B_{\rho}}|\grad\vt|^p\wt(\xi)\,d\xi\simeq\int_{G_{t-1}}|\grad v|^p\,dx.
\end{align*}
Substituting the last integral into \eqref{eq-Cp-int} gives
\[
C_p(E)\simle \int_{G_{t-1}}|\grad v|^p\,dx. 
\]
Taking the infimum over of all $v\in C_0^\infty(\clG\setm Q_{t-1})$ admissible 
in the definition of $\cp_{p,G_{t-1}}(E)$ completes the proof. 
\end{proof}

\section{Boundary regularity at $\infty$}
\label{sect-bdry-reg}

The solution $u$ of the mixed boundary value problem~\eqref{eq-weak-sol-p-Lapl},
obtained in Section~\ref{sec-existence}, is 
continuous in $G\setm F$ and at the Neumann boundary~$\bdy G\setm F$. 
If $f\in C(F_0)$, then $u$ is also continuous at the Dirichlet boundary~$F_0$,
except for a set of zero $\Cp$-capacity.
We shall now study its continuity at~$\infty$.

Recall that $F$ is a closed subset of $\clG$, containing the base
$B'\times\{0\}$, and that $F_0= F\cap\bdy(G\setm F)$ is the Dirichlet
boundary of $G\setm F$.

If $F$ is bounded, then the solution $\ut\in \Hp\loc(D,\wt)$ of \eqref{eqnB} in
$D$, constructed in the proofs of
Theorems~\ref{thm-ex-sol-Lip} and~\ref{thm-sol-Lip}, 
belongs to $\Hp(D,\wt)$ (when $f\in\Llp_\ka(G\setm F)$)
or is bounded (when $f\in C(F_0)$).
Since the origin $0$ has zero $(p,\wt)$-capacity by
Lemma~\ref{lem-cap-B-2B},
the removability results \cite[Lemma~7.33 and Theorem~7.36]{HKM} 
imply that $\ut$ is a solution of \eqref{eqnB} in $\Om=D\cup\{0\}$.
Consequently, $\ut$ is continuous at $\xi=0$ and it follows that the limit
$\lim_{\clG\setm F\ni x\to\infty}u(x)$ always exists when $F$ is bounded.

\begin{defi} \label{defi-infty-reg}
Assume that $F$ is unbounded.
We say that the point at $\infty$ is \emph{regular} for the mixed boundary 
value problem~\eqref{eq-weak-sol-p-Lapl} in $G\setm F$
with zero Neumann data on $\bdy G\setm F$ 
if for all Dirichlet boundary data $f\in C(F_0)$ with a finite limit
\begin{equation}  \label{eq-f(infty)-ex}
\lim_{\substack{x_n\to\infty\\ x\in F_0}} f(x) =:f(\infty), 
\end{equation}
the continuous solution $u$, provided by 
Theorem~\ref{thm-cont}, satisfies
\begin{equation}   \label{eq-def-reg-infty}
\lim_{\substack{x_n\to\infty\\ x\in G\setm F}} u(x) =f(\infty).
\end{equation}
\end{defi}

As before, we will study the mixed boundary value 
problem~\eqref{eq-weak-sol-p-Lapl} on  $G\setm F$
 by means of the weighted equation $\dvg\B(\xi,\grad\ut(\xi))=0$
in the bounded domain $D=B_1\setm (T(F)\cup PT(F)\cup\{0\})$.

\begin{lem}		\label{lem-reg-infty-0}
The point at $\infty$ is regular for the mixed boundary 
value problem~\eqref{eq-weak-sol-p-Lapl} in $G\setm F$ if and only if 
the origin $0\in\bdry D$ is regular with respect to the equation
\begin{equation}   \label{eq-dvg-B-in-D}
\dvg\B(\xi,\grad\ut(\xi))=0 \quad \text{in } D,
\end{equation}
where $\A$ is as in \eqref{eq-def-B}.
\end{lem}

Recall that regularity with respect to \eqref{eq-dvg-B-in-D} is defined in 
Heinonen--Kilpel\"ainen--Martio~\cite[p.~122]{HKM}
using~\eqref{eq-def-reg-HKM} for all $\fh\in\Hp(D,\wt)\cap C(\clD)$.

\begin{proof}
First, assume that $0$ is regular with respect to the 
equation~\eqref{eq-dvg-B-in-D} and let $f\in C(F_0)$ be such 
that the limit in~\eqref{eq-f(infty)-ex} is finite.
Also, assume without loss of generality that $f(\infty)=0$.

Given $\eps>0$, let $f_0$ be a compactly supported Lipschitz function on $\clG$ 
such that $|f_0-f|\le\eps/2$ on $F_0$.
Define $\ft_0$ as in \eqref{eq-fk-Lip}.
Then $\ft_0\in\Hp(D,\wt)\cap C(\clD)$ 
vanishes in a neighbourhood of the origin $\xi=0$.
By \cite[Theorems~3.17 and 3.70]{HKM}, there is a unique continuous 
solution $\vt\in\Hp(D,\wt)$ of \eqref{eq-dvg-B-in-D} in $D$ 
such that $\vt-\ft_0\in\Hp_0(D,\wt)$.

Since $0\in \bdy D$ is regular, there is $\de>0$ such that $|\vt(\xi)|<\eps/2$ 
whenever $\xi\in D$ and $|\xi|\le \de$.
Let $u$ be the bounded continuous solution of \eqref{eq-weak-sol-p-Lapl} 
with Dirichlet boundary data $f$, constructed in
Theorem~\ref{thm-sol-Lip}. 
Since $|f_0-f|<\eps/2$ on $F_0$, Corollary~\ref{cor-com-princ}
implies that
\[
\vt\circ T-\frac{\eps}{2}\le u(x)   \le \vt\circ T+\frac{\eps}{2}
\quad  \text{in } G\setm F,
\]
so that for $x\in G\setm F$ with $x_n>\ka^{-1}\log(1/\de)$ we have $|u(x)|<\eps$. 
Since $\eps>0$ was chosen arbitrarily, we conclude that
\eqref{eq-def-reg-infty} holds.

Conversely, assume that $\infty$ is regular for the mixed boundary value problem 
\eqref{eq-weak-sol-p-Lapl} and let $\fh\in\Hp(D,\wt)\cap C(\clD)$ be 
arbitrary.  
We shall show that the solution $\uh$ of~\eqref{eq-dvg-B-in-D}
in $D$ with boundary data $\fh$ satisfies $\lim_{\xi\to0} \uh(\xi)=\fh(0)$.

The function $\fh$ need not necessarily be symmetric (by which we mean 
$\fh=\fh\circ P$), so we instead consider $\fh_1=\min\{\fh,\fh\circ P\}$ 
and $\fh_2=\max\{\fh,\fh\circ P\}$ which are symmetric in the above sense.
By Theorem~1.20 in Heinonen--Kilpel\"ainen--Martio~\cite{HKM}, 
$\fh_1,\fh_2\in\Hp(D,\wt)\cap C(\clD)$.
By Theorems~3.17 and~3.70 in \cite{HKM}, there exist bounded continuous 
weak solutions $\uh_1,\uh_2\in\Hp(D,\wt)$  of \eqref{eq-dvg-B-in-D} 
in $D$ with boundary data $\fh_1,\fh_2$ respectively.
As in Section~\ref{sec-existence}, the functions $u_1=\uh_1\circ T$ 
and $u_2=\uh_2\circ T$  are solutions of
\eqref{eq-weak-sol-p-Lapl} in $G\setm F$ with zero Neumann boundary data 
and Dirichlet boundary data $f_1=\fh_1\circ T$ and $f_2=\fh_2\circ T$ 
in $\Llp_k(G\setm F)$.
Note that $f_1$ and $f_2$ are continuous on $F_0$ and that 
\[
\lim_{\substack{x_n\to\infty\\ x\in F_0}}f_j(x)=\fh(0),\quad j=1,2.
\]
Since $\infty$ is regular for \eqref{eq-weak-sol-p-Lapl} in $G\setm F$, 
the solutions $u_j$ satisfy
\[
\lim_{\substack{x_n\to\infty\\ x\in G\setm F}}u_j(x)=\fh(0), \quad j=1,2.
\]
Finally, since $\uh_j=\uh_j\circ P$ by Corollary~\ref{cor-uop}, it follows that
\[
\lim_{\xi\to0}\uh_j(\xi) =\lim_{T(G\setm F)\ni\xi\to0}\uh_j(\xi)
=\fh(0), \quad j=1,2.
\]  
Note that $\fh_1\le\fh\le\fh_2$ and hence $\uh_1\le\uh\le\uh_2$, from which we conclude that 
\[
\lim_{\xi\to0}\uh(\xi)=\fh(0). \qedhere
\] 
\end{proof}

Regular boundary points for \eqref{eq-dvg-B-in-D} are characterized by the following
\emph{Wiener criterion}, 
see Heinonen--Kilpel\"ainen--Martio~\cite[Theorem~21.30\,(i)$\eqv$(v)]{HKM}  
and Mikkonen~\cite{M}.
Because of Lemma~\ref{lem-cap-B-2B}, the Wiener criterion 
\begin{equation}			\label{eq-Wiener-crit}
\int_0^1\biggl(\frac{\cp_{p,\wt}(\Ft\cap B_r,B_{2r})}{\cp_{p,\wt}(B_r,B_{2r})}
    \biggr)^{1/(p-1)}\,\frac{dr}{r}=\infty
\end{equation}
 at $0\in\bdry D$ reduces to 
\begin{equation}   \label{eq-Wiener-simpler}
\int_0^1 \cp_{p,\wt}(\Ft\cap B_r,B_{2r})^{1/(p-1)} \,\frac{dr}{r} = \infty.
\end{equation}
Since the quotient in \eqref{eq-Wiener-crit} is at most $1$ and
$\int_\tau^1 r^{-1}dr =\log(1/\tau)$ is finite, we also see that the integral $\int_0^1$ 
in the Wiener criterion can equivalently be replaced by $\int_0^\tau$ for any $\tau>0$.
Moreover,
\[
\cp_{p,\wt}(\Ft\cap B_r,B_{2r})\ge\cp_{p,\wt}(\Ft\cap\clB_{r/2},B_{2r})
\simge\cp_{p,\wt}(\Ft\cap\clB_{r/2},B_r),
\]
where the last inequality follows as in the proof of \cite[Lemma~2.16]{HKM}.
Inserting this into \eqref{eq-Wiener-simpler}
shows that \eqref{eq-Wiener-simpler} is equivalent to 
\begin{equation}   \label{eq-Wiener-clB}
\int_0^1 \cp_{p,\wt}(\Ft\cap\clB_\rho,B_{2\rho})^{1/(p-1)} \,\frac{d\rho}{\rho} = \infty.
\end{equation}
We will now further rewrite this condition to better match the transformation $T^{-1}$
back to the cylinder $G$.

\begin{lem} \label{lem-Wiener-al}
For any $\al> 2^{1/(p-1)}$,
condition~\eqref{eq-Wiener-simpler} is equivalent to 
\begin{equation}   \label{eq-Wiener-al}
\int_0^1 \cp_{p,\wt}(\Ft\cap(B_r\setm B_{r^\al}),B_{2r})^{1/(p-1)} \,\frac{dr}{r} 
       = \infty.
\end{equation}
\end{lem}

\begin{proof}
One implication is clear since the integral in~\eqref{eq-Wiener-simpler} 
majorizes the one in~\eqref{eq-Wiener-al}.

Conversely, use the subadditivity
\begin{equation}   \label{eq-subadd-w-al}
\cp_{p,\wt}(\Ft\cap B_r,B_{2r}) 
   \le \cp_{p,\wt}(\Ft\cap(B_r\setm B_{r^\al}),B_{2r}) 
               + \cp_{p,\wt}(\Ft\cap B_{r^\al},B_{2r}) 
\end{equation}
to majorize the integral in~\eqref{eq-Wiener-simpler} by 
a sum of two integrals, 
one for each of the sets in the right-hand side of \eqref{eq-subadd-w-al} as follows.
For all $0<\de<1$, we have 
\begin{align}   \label{int-subadd-w-al}
&\int_\de^1\cp_{p,\wt}(\Ft\cap B_r,B_{2r})^{1/(p-1)}\,\frac{dr}{r} \nonumber\\
&\qquad \quad  \le 
2^{1/(p-1)}\int_\de^1\cp_{p,\wt}(\Ft\cap(B_r\setm B_{r^\al}),B_{2r})^{1/(p-1)}\,\frac{dr}{r} \\  
&\qquad \qquad 
+ 2^{1/(p-1)}\int_\de^1\cp_{p,\wt}(\Ft\cap B_{r^\al},B_{2r})^{1/(p-1)}\,\frac{dr}{r}.  \nonumber
\end{align}
The first integral on the right-hand side of \eqref{int-subadd-w-al} is for all 
$\de>0$ majorized by the integral in \eqref{eq-Wiener-al}. 
For the second integral, we use the change of variables $\rho=r^\al$,
together with the fact that
\[
\cp_{p,\wt}(\Ft\cap B_\rho,B_{2\rho^{1/\al}})\le\cp_{p,\wt}(\Ft\cap B_\rho,B_{2\rho}),
\quad 0<\rho<1,
\]
and estimate it as 
\begin{align}   \label{int-cpt-w-al}
\int_\de^1\cp_{p,\wt}(\Ft\cap B_{r^\al},B_{2r})^{1/(p-1)}\,\frac{dr}{r} 
   &\le \frac{1}{\al}\int_{\de^\al}^\de\cp_{p,\wt}(\Ft\cap B_\rho,B_{2\rho^{1/\al}})^{1/(p-1)}
           \,\frac{d\rho}{\rho} \nonumber\\
   & + \frac{1}{\al}\int_\de^1\cp_{p,\wt}(\Ft\cap B_\rho,B_{2\rho})^{1/(p-1)}
           \,\frac{d\rho}{\rho}.
\end{align}
For $\al> 2^{1/(p-1)}$, the last integral in \eqref{int-cpt-w-al} can be 
subtracted from the integral on the left-hand side of \eqref{int-subadd-w-al}. 
The remaining integral on the right-hand side in \eqref{int-cpt-w-al} is estimated using 
\eqref{eq-cap-cond}  as follows,
\begin{align*}
\int_{\de^\al}^\de \cp_{p,\wt}(\Ft\cap B_\rho,B_{2\rho^{1/\al}})^{1/(p-1)} 
           \,\frac{d\rho}{\rho}
&\le \int_{\de^\al}^\de \cp_{p,\wt}(B_\rho,B_{\rho^{1/\al}})^{1/(p-1)} 
           \,\frac{d\rho}{\rho}\\
&\simle \int_{\de^\al}^\de 
\biggl(\log\biggl(\frac{\rho^{1/\al}}{\rho}\biggr)\biggr)^{-1} \,\frac{d\rho}{\rho}
         =\frac{\al}{\al-1}\log\al.       
\end{align*}
Inserting this into \eqref{int-cpt-w-al} and \eqref{int-subadd-w-al}, 
together with letting $\de\to 0$, shows that 
\begin{align*}
&\biggl(1-\frac{2^{1/(p-1)}}{\al}\biggr)\int_0^1
         \cp_{p,\wt}(\Ft\cap B_r,B_{2r})^{1/(p-1)}\,\frac{dr}{r} \\
&\qquad \quad  \le 2^{1/(1-p)} \int_0^1
     \cp_{p,\wt}(\Ft\cap(B_r\setm B_{r^\al}),B_{2r})^{1/(p-1)}\,\frac{dr}{r}  
+ \frac{C\log\al}{\al-1},
\end{align*}
with $C$  independent of $\al$.
We therefore conclude that the integral in 
\eqref{eq-Wiener-simpler} is finite whenever the one in \eqref{eq-Wiener-al} 
is finite. 
Thus \eqref{eq-Wiener-simpler} and \eqref{eq-Wiener-al} are equivalent.\qedhere
\end{proof}

\begin{lem} \label{lem-Bral-Br2}
The condition \eqref{eq-Wiener-al} is equivalent to
\begin{equation}		\label{eq-int-with-r2}
\int_0^1 \cp_{p,\wt}(\Ft\cap(\clB_r\setm B_{r^2}),B_{2r})^{1/(p-1)} \,\frac{dr}{r}=\infty.
\end{equation}
\end{lem}
\begin{proof}
We already know that \eqref{eq-Wiener-al} is equivalent to \eqref{eq-Wiener-simpler} 
and thus to \eqref{eq-Wiener-clB}.
The integral in \eqref{eq-Wiener-clB} clearly majorizes the integral 
in \eqref{eq-int-with-r2} and so \eqref{eq-int-with-r2} implies \eqref{eq-Wiener-al}.

Conversely, we can without loss of generality assume that $\al=2^m$ for some integer
$m\ge1$, and hence $r^\al\le r^2$.
By the subadditivity and monotonicity of $\cp_{p,\wt}$, see
\cite[Theorem~2.2]{HKM}, we have for all $0<r<1$,
\begin{equation}   \label{eq-split-al-sum}
\cp_{p,\wt}(\Ft\cap(B_r\setm B_{r^\al}),B_{2r})
\le \sum_{k=0}^{m-1} \cp_{p,\wt}\bigl(\Ft\cap(B_{r^{2^k}}\setm 
    B_{(r^{2^k})^2}),B_{2r^{2^k}}\bigr).
\end{equation}
For each $k=0,\ldots,m-1$, the change of variables $\rho=r^{2^k}$ implies that
\begin{align*}
&\int_0^1 \cp_{p,\wt} \bigl( \Ft\cap \bigl( B_{r^{2^k}}\setm 
    B_{(r^{2^k})^2} \bigr),B_{2r^{2^k}} \bigr)^{1/(p-1)} \,\frac{dr}{r} \\
& \qquad \quad = 2^{-k}\int_0^1 \cp_{p,\wt}(\Ft\cap(B_\rho\setm B_{\rho^2}),B_{2\rho})^{1/(p-1)} \,\frac{d\rho}{\rho}.
\end{align*}
Together with \eqref{eq-split-al-sum}, this yields
\begin{align*}
&
\int_0^1 \cp_{p,\wt}\bigl(\Ft\cap\bigl(B_r\setm B_{r^\al}\bigr),B_{2r}\bigr)^{1/(p-1)} \,\frac{dr}{r} \\
&\qquad \quad \simle \sum_{k=0}^{m-1}\int_0^1\cp_{p,\wt}\bigl(\Ft\cap\bigl(B_{r^{2^k}}\setm 
    B_{(r^{2^k})^2}\bigr),B_{2r^{2^k}}\bigr)^{1/(p-1)}\,\frac{dr}{r}\\
    &\qquad \quad \simle\sum_{k=0}^{m-1}2^{-k}\int_0^1 
          \cp_{p,\wt}(\Ft\cap(B_\rho\setm B_{\rho^2}),B_{2\rho})^{1/(p-1)} \,\frac{d\rho}{\rho}\\
& \qquad \quad \le 2 \int_0^1 
    \cp_{p,\wt}(\Ft\cap(\clB_r\setm B_{r^2}),B_{2r})^{1/(p-1)} \,\frac{dr}{r},
\end{align*}
which shows that \eqref{eq-Wiener-al} implies \eqref{eq-int-with-r2}.
\end{proof}

Finally, we prove the following concrete Wiener criterion for the boundary 
regularity at $\infty$ for the mixed boundary value problem \eqref{eq-weak-sol-p-Lapl} 
in $G\setm F$, stated as \eqref{eq-Wiener-intro} in the introduction.

\begin{thm} \label{thm-reg-infty}
The point at $\infty$ is regular for the mixed boundary value problem 
\eqref{eq-weak-sol-p-Lapl}   
in $G\setm F$ if and only if the following condition holds
\begin{equation}		\label{eq-infty}
\int_1^\infty\cp_{p,G_{t-1}}(F\cap(\clG_t\setm G_{2t}))^{1/(p-1)}\,dt=\infty.
\end{equation}
\end{thm}

\begin{proof}
Lemma~\ref{lem-reg-infty-0} guarantees that the point at $\infty$ is regular for 
\eqref{eq-weak-sol-p-Lapl}   
in $G\setm F$ if and only if the origin $0\in\bdy D$ is regular for \eqref{eq-dvg-B-in-D}
in $D$. 
Regular points for \eqref{eq-dvg-B-in-D} are characterized
by the Wiener criterion~\eqref{eq-Wiener-crit}.

By Lemma~\ref{lem-cap-B-2B}, the Wiener criterion at $0\in\bdy D$ reduces 
to~\eqref{eq-Wiener-simpler}.
Lemma~\ref{lem-Wiener-al} shows that \eqref{eq-Wiener-simpler} is  
for $\al>2^{1/(p-1)}$ equivalent to \eqref{eq-Wiener-al}, which is 
in turn equivalent to \eqref{eq-int-with-r2}, by Lemma~\ref{lem-Bral-Br2}. 
What now remains is to show that the integral in \eqref{eq-int-with-r2} 
diverges if and only if the one in \eqref{eq-infty} does.

As in \cite[Lemma~2.16]{HKM}, it can be shown that replacing the ball 
$B_{2r}$ in \eqref{eq-int-with-r2} by $B_{\la r}$, for any $\la>1$, results
in an integral that is comparable to the one in \eqref{eq-int-with-r2}.
The convergence of the integral is thus not influenced by the change to 
$B_{\la r}$. 
In particular, we can take $\la=e^{\ka}$.
Moreover, since by Lemma~\ref{lem-cap-B-2B},
\[
\int_{e^{-\ka}}^1 \cp_{p,\wt}(\Ft\cap(\clB_r\setm B_{r^2}),B_{e^{\ka}r})^{1/(p-1)} 
     \,\frac{dr}{r}\simle\int_{e^{-\ka}}^1\,\frac{dr}{r}=\ka<\infty,
\]
integrating only over $0<r\le e^{-\ka}$ does not play a role for
the convergence of the integral either. 
Thus the regularity of the origin $0\in\bdy D$ is equivalent to 
\begin{equation}  \label{eq-reg-infty}
\int_0^{e^{-\ka}} \cp_{p,\wt}(\Ft\cap(\clB_r\setm B_{r^2}),B_{e^{\ka}r})^{1/(p-1)} 
     \,\frac{dr}{r}=\infty.
\end{equation}

To finish the proof, apply the change of variables $r=e^{-\ka t}$ and the fact that the ball $B_r$, 
$0<r\le e^{-\ka}$, corresponds to the truncated cylinder
\[
G_t:=\{x\in\clG:x_n>t\},\quad \text{with } t=-\frac{1}{\ka}\log r \ge1.
\]
Also note that with this notation, $B_{r^2}$ and $B_{e^\ka r}$ correspond to $G_{2t}$ and $G_{t-1}$, 
respectively.
Comparing $\cp_{p,\wt}$ and $\cp_{p,G_{t-1}}$ using Lemma~\ref{lem-compare-cap}, 
it therefore follows that the regularity condition \eqref{eq-reg-infty} is equivalent to 
 \begin{equation*}
\int_1^\infty\cp_{p,G_{t-1}}(F\cap(\clG_t\setm G_{2t}))^{1/(p-1)}\,dt=\infty.\qedhere
\end{equation*} 
\end{proof}

\begin{remark} \label{rmk-sum}
The condition \eqref{eq-infty} is clearly equivalent to the condition
\[
\sum_{j=1}^\infty\cp_{p,G_{j-1}}(F\cap(\clG_j\setm G_{2j}))^{1/(p-1)}=\infty.
\]
\end{remark}

We end the paper with the following two examples illustrating when $\infty$ 
is irregular and regular for the mixed boundary value problem \eqref{eq-weak-sol-p-Lapl} 
in $G\setm F$.

\begin{example}
Assume that $1<p<n$.
For $i=1,2,\cdots,$ define $B^i:=B(z^i,2^{-i})$ with
fixed $z^i\in G_i\setm G_{i+1}$.
Let
\begin{equation}\label{eq-F-Bi}
F=\bigl(\clBprime \times\{0\}\bigr) \cup 
     \biggl(\bigcup_{i=1}^\infty  \clBi  \cap\clG\biggr).
\end{equation}
Then by Lemma~\ref{lem-cap-Cp}, we have for $j=1,2,\cdots$ that
\begin{equation} \label{eq-cap-clB}
\cp_{p,G_{j-1}}(F\cap(\clG_j\setm G_{2j}))
\simle C_p(F\cap(\clG_j\setm G_{2j}))
\le C_p\biggl(\bigcup_{i=j}^{2j-1}\clBi\biggr).
\end{equation}
Applying the subadditivity property of $C_p$ to the last expression 
in~\eqref{eq-cap-clB} and then using Corollary~2.41 from 
Heinonen--Kilpel\"ainen--Martio~\cite{HKM}, we get that
\[
\cp_{p,G_{j-1}}(F\cap(\clG_j\setm G_{2j}))
\simle\sum_{i=j}^{2j-1}C_p \bigl(\clBi\bigr)
\simeq\sum_{i=j}^{2j-1}2^{-i(n-p)}\simle 2^{-j(n-p)},
\]
with comparison constants depending only on $n$ and $p$.
Thus,
\[
\sum_{j=1}^\infty\cp_{p,G_{j-1}}(F\cap(\clG_j\setm G_{2j}))^{1/(p-1)}
  \simle\sum_{j=1}^\infty 2^{-j(n-p)/(p-1)}   
<\infty.
\]
Theorem~\ref{thm-reg-infty} and Remark~\ref{rmk-sum} therefore imply 
that $\infty$ is irregular for the mixed boundary value problem 
\eqref{eq-weak-sol-p-Lapl}    
in $G\setm F$ with $F$ defined as in \eqref{eq-F-Bi}.
\end{example}

\begin{example}
Let $p>1$ and
\begin{equation}\label{eq-F-Kj}
F=\bigl(\clBprime\times\{0\}\bigr) \cup  \bigcup_{j=1}^\infty K_j,
\end{equation}
 where $K_j\subset \clG_j\setm G_{j+1}$ are closed sets such that 
\[
C_p(K_j)\ge\de>0,\quad j=1,2,\cdots.
\]
Then by Lemma~\ref{lem-Cp-cap},
\[
\sum_{j=1}^\infty\cp_{p,G_{j-1}}(F\cap(\clG_j\setm G_{2j}))^{1/(p-1)}
\simge \sum_{j=1}^\infty C_p(K_j)^{1/(p-1)}=\infty.
\] 
Theorem~\ref{thm-reg-infty} and Remark~\ref{rmk-sum} show that $\infty$ 
is regular for the mixed boundary value problem \eqref{eq-weak-sol-p-Lapl}    
in $G\setm F$ with $F$ defined as in \eqref{eq-F-Kj}.

A particular example of this situation is $K_j=E\times[j,j+1]$,
where $E\subset\clBprime$ is any nonempty closed set 
of Hausdorff dimension 
\[
\dim_H(E)>n-1-p.
\]
Then $\dim_H(E\times[0,1])>n-p$, which implies that 
\[
C_p(K_j)=C_p(E\times[0,1])>0,
\] 
by e.g.\ Ziemer~\cite[Remark~2.6.15 and Theorem~2.6.16]{Zie}.
If $p>n$ then singletons have positive capacity and it is thus
sufficient for regularity at infinity that $F$ is unbounded.
\end{example}

\end{document}